\documentclass[12pt]{article}
\usepackage[utf8]{inputenc}
\usepackage[english]{babel}
\usepackage[a4paper]{geometry}
\usepackage{amsfonts,amssymb,amsmath,amsthm,hyperref,graphicx,enumerate,wasysym,xcolor,fancyvrb}
\usepackage{titlesec}
\usepackage{mathtools,mathrsfs}
\usepackage[nottoc,numbib]{tocbibind}
\usepackage{chngcntr}
\usepackage[safeinputenc,maxbibnames=99]{biblatex}
\usepackage{csquotes}
\usepackage{tikz} 
\usetikzlibrary[shapes.arrows]

\numberwithin{equation}{section}

\def\R{{\mathbb R}}

\def\N{{\mathbb N}}

\def\CC{\mathcal{C}}

\def\CE{\mathcal{E}}

\def\CH{\mathcal{H}}
\def\CP{\mathcal{P}}

\def\cat{\mathrm{cat}}
\def\vol{\mathrm{vol}}
\def\eps{\varepsilon}
\def\loc{\mathrm{loc}}
\def\supp{\mathrm{supp}}

\def\dist{\mathrm{dist}}

\titleclass{\subsubsubsection}{straight}[\subsection]

\newcounter{subsubsubsection}[subsubsection]
\renewcommand\thesubsubsubsection{\thesubsubsection.\arabic{subsubsubsection}}
\renewcommand\theparagraph{\thesubsubsubsection.\arabic{paragraph}} 

\titleformat{\subsubsubsection}
{\normalfont\normalsize\bfseries}{\thesubsubsubsection}{1em}{}
\titlespacing*{\subsubsubsection}
{0pt}{3.25ex plus 1ex minus .2ex}{1.5ex plus .2ex}

\makeatletter
\renewcommand\paragraph{\@startsection{paragraph}{5}{\z@}%
	{3.25ex \@plus1ex \@minus.2ex}%
	{-1em}%
{\normalfont\normalsize\bfseries}}
\renewcommand\subparagraph{\@startsection{subparagraph}{6}{\parindent}%
	{3.25ex \@plus1ex \@minus .2ex}%
	{-1em}%
{\normalfont\normalsize\bfseries}}
\def\toclevel@subsubsubsection{4}
\def\toclevel@paragraph{5}
\def\toclevel@paragraph{6}
\def\l@subsubsubsection{\@dottedtocline{4}{7em}{4em}}
\def\l@paragraph{\@dottedtocline{5}{10em}{5em}}
\def\l@subparagraph{\@dottedtocline{6}{14em}{6em}}
\makeatother

\theoremstyle{plain}\newtheorem{theorem}{Theorem}
\theoremstyle{plain}\newtheorem{proposition}[theorem]{Proposition}
\theoremstyle{plain}\newtheorem{lemma}[theorem]{Lemma}
\theoremstyle{plain}
\theoremstyle{plain}

\theoremstyle{definition}\newtheorem{definition}[theorem]{Definition}
\theoremstyle{remark}\newtheorem{remark}[theorem]{Remark}
\theoremstyle{definition}\newtheorem{example}[theorem]{Example}

\numberwithin{theorem}{section} 
\numberwithin{equation}{section}

\theoremstyle{plain}\newtheorem{theoremletter}{Theorem}

\numberwithin{theorem}{section} 
\numberwithin{equation}{section}

\newtheorem{asu}{}

\usepackage{color}
\newcommand{\de}{\mathrm{d}}

\DeclarePairedDelimiter{\norm}{\lVert}{\rVert}

\title{Multiplicity results
	for mass constrained Allen-Cahn equations
on Riemannian manifolds with boundary}
\author{D. Corona, S. Nardulli, R. Oliver-Bonafoux, G. Orlandi and P. Piccione}

\begin{document}
\maketitle
\begin{abstract}
    We present multiplicity results for mass constrained Allen-Cahn
    equations on a Riemannian manifold with boundary,
    considering both Neumann and Dirichlet conditions.
    These results hold under the assumptions of small mass constraint and small diffusion parameter.
    We obtain lower bounds on the number of solutions according to the Lusternik--Schnirelmann category of the manifold in case of Dirichlet boundary conditions and of its boundary in the case of Neumann boundary conditions.
    Under generic non-degeneracy assumptions on the solutions, we obtain stronger results based on Morse inequalities.
    Our approach combines topological and variational methods with tools from Geometric Measure Theory.
\end{abstract}

\emph{2020 Mathematics Subject Classification:} 35A15, 49Q20, 58J32.

\emph{Key words and phrases:} Variational methods, $\Gamma$-convergence, isoperimetric-type problems, Allen-Cahn equations, manifolds with boundary.

\section{Introduction}

Let $(M,g)$ be a smooth, compact $n$--dimensional
Riemannian
manifold with boundary, for $n\ge 2$.
Let $W \in \CC^3_{\loc}(\R,[0,+\infty))$ a double-well potential
such  that $W^{-1}(0)=\{0,1\}$. For fixed $\eps,m > 0$, consider the following Allen-Cahn type equations with a mass constraint:
\begin{equation}
	\label{eq:Cahn-Hilliard-PDE}
	\begin{dcases}
		-\eps \Delta u_{\eps,m}+ \frac{1}{\eps} W'(u_{\eps,m})=\lambda_{\eps,m}, & \mbox{ on } M,\\
		\int_M u_{\eps,m} \de v_g=m, &
	\end{dcases} 
\end{equation}
where the unknown parameter $\lambda_{\eps,m} \in \R$ is a Lagrange multiplier
related to the mass constraint $\int_M u_{\eps,m} \de v_g=m$. We study \eqref{eq:Cahn-Hilliard-PDE} under both Neumann and Dirichlet boundary conditions,
which leads respectively to the two following problems:
\begin{equation}
	\label{eq:PDE_Neumann}
	\begin{dcases}
		-\eps \Delta u_{\eps,m}+ \frac{1}{\eps} W'(u_{\eps,m})=\lambda_{\eps,m}, & \mbox{ on } M,\\
		\int_M u_{\eps,m}=m, & \\
		\frac{\partial u_{\eps,m}}{\partial \nu}=0, & \mbox{ on } \partial M,
	\end{dcases} 
\end{equation}
where $\nu\in TM\big|_{\partial M}$ is the unit inner normal vector to $\partial M$ in $M$,
and
\begin{equation}
	\label{eq:PDE_Dirichlet}
	\begin{dcases}
		-\eps \Delta u_{\eps,m}+ \frac{1}{\eps} W'(u_{\eps,m})=\lambda_{\eps,m}, & \mbox{ on } M,\\
		\int_M u_{\eps,m}=m, & \\
		u_{\varepsilon,m} = 0, & \mbox{ on } \partial M.
	\end{dcases} 
\end{equation}

A straightforward computation shows that solutions
of the Neumann problem \eqref{eq:PDE_Neumann} can be found among 
the critical points of the functional
\begin{equation}
	\label{eq:def-E-varepsilon}
	\mathcal{E}_{\varepsilon}(u)\coloneqq 
	\int_M \left(\varepsilon\, \frac{\lvert \nabla u \rvert^2}{2}+\frac{1}{\eps}W(u) \right)\de v_g,
\end{equation}
in the function space 
\begin{equation*}
	\mathcal{H}_m\coloneqq 
	\left\{u \in H^1(M,\R):\int_M u\, \de v_g=m \right\},
\end{equation*}
while critical points of the restriction of $\CE_\eps$ to the subspace
\begin{equation*}
	\mathcal{H}_{m,0}\coloneqq 
	\left\{u \in H^1_0(M,\R):\int_M u\, \de v_g=m \right\},
\end{equation*}
are solutions of the Dirichlet problem \eqref{eq:PDE_Dirichlet}.

The functional $\CE_\eps$ is widely used in the modelling of phase transition phenomena. For small $\eps>0$, critical points of $\CE_\eps$ in $\CH_m$ and $\CH_{m,0}$ develop transition layers whose profiles are close to solutions of isoperimetric-type problems. This link was first noticed (in the minimizing case, using the ideas of $\Gamma$-convergence) in the seminal work of Modica \cite{modica} in the case of Neumann boundary conditions and by Owen, Rubinstein and Sternberg \cite{owen-rubinstein-sternberg} in the case of Dirichlet boundary conditions. Extensions to arbitrary critical points have been given by Hutchinson and Tonegawa \cite{hutchinson-tonegawa}. Analogously, in the absence of a mass constraint critical points of Allen-Cahn functionals are linked to the min-max theory of minimal surfaces (and, more generally, submanifolds) which was initiated by Almgren \cite{almgren62,almgren65} and improved later by Pitts \cite{pitts81}.
This theory has been object of an extensive study in the past years after the works of Marques and Neves (e.g., \cite{marques-neves14,marques-neves17}). In particular, a min-max theory for isoperimetric-type surfaces was developed by Zhou and Zhu \cite{zhou-zhu2019}.
As a consequence, there has been an increased and renewed interest in the study of existence and multiplicity of critical points for functionals of Allen-Cahn type and their asymptotic behavior as $\eps \to 0$, departing from results of Gaspar and Guaraco \cite{gaspar-guaraco,guaraco}.
In the presence of a mass constraint, first studies in this direction have been carried out in
\cite{benci-nardulli-osorio-piccione} in the case of manifolds without boundary, while related vectorial problems have been recently studied in \cite{andrade-conrado-nardulli-piccione-resende}.
In a related direction, Bellettini and Wickramaseckera \cite{bellettini-wickramasekera} studied the existence of solutions for a more general equation of Allen-Cahn-type associated to surfaces of prescribed mean curvature.

The purpose of this paper is to extend
the main results of \cite{benci-nardulli-osorio-piccione}
to the case of manifolds with boundary.
More precisely, we establish lower bounds on
the number of solutions for~\eqref{eq:PDE_Neumann}
and~\eqref{eq:PDE_Dirichlet} in terms of certain topological invariants of the underlying manifold.
These lower bounds hold true for values
of the parameters $\eps$ and $m$ which are relatively small.
We believe that these results on manifolds with boundary are of interest because they include the case of  Euclidean domains in $\R^n$ (the one originally considered in classical works such as \cite{modica,owen-rubinstein-sternberg}), which is left out in the setting of manifolds without boundary. Moreover, they require a fine understanding of relative isoperimetric-type problems for small volumes on Riemannian manifolds with boundary, which actually depend heavily on the boundary conditions (Dirichlet or Neumann) that are imposed. Overall, significant differences appear when one moves from the setting without boundary to the setting with boundary.

In order to state and prove our main results, we consider the following typical assumptions on the potential $W$:
\begin{asu}
	\label{asu_nondegenerate}
	The wells are non-degenerate global minimizers,
	i.e. $W''(0),W''(1) > 0$.
\end{asu}
\begin{asu}\label{asu_coercive}
	The potential $W$ is coercive, i.e. there exist $R>0$ and $\alpha>0$ such that
	\begin{equation}\label{coercive_V}
		W'(u) u \geq \alpha \lvert u \rvert^2,
		\quad \text{ if } \lvert u \rvert \geq R.
	\end{equation}
\end{asu}
\begin{asu}\label{asu_subcritical} The potential $W$ has subcritical growth at infinity, i.e.
	there exist $p \in [2,2^*)$, $\alpha'>0$ and $R'>0$
	such that 
	\begin{equation*}
		\lvert W''(u) \rvert \leq \alpha' \lvert u \rvert^{p-2},
		\quad \text{ if } \lvert u \rvert \geq R'.
	\end{equation*}
 Here $2^*$ denotes the critical Sobolev exponent, i.e. $2^*=2n/(n-2)$ if $n \geq 3$ and $2^*=+\infty$ if $n=2$.
\end{asu}
For $n=2,3$,
a standard example of function $W:\R\to [0,+\infty)$ satisfying
\ref{asu_nondegenerate}, \ref{asu_coercive} and \ref{asu_subcritical}
is the quartic potential:
\begin{equation*}
    W_{q}(u)=u^2(u-1)^2.
\end{equation*}

We will denote by $\mathcal{E}_{\varepsilon,m}$
the restriction of $\mathcal{E}_{\varepsilon}$ to $\mathcal{H}_{m}$,
and by $\overline{\mathcal{E}}_{\varepsilon,m}$
the restriction to $\mathcal{H}_{m,0}$,
that is
\[
	\mathcal{E}_{\varepsilon,m} \coloneqq \mathcal{E}_{\varepsilon}\big|_{\mathcal{H}_{m}}
	\quad\text{and}\quad
	\overline{\mathcal{E}}_{\varepsilon,m} \coloneqq \mathcal{E}_{\varepsilon}\big|_{\mathcal{H}_{m,0}}.
\]

Our multiplicity results for critical points of $\mathcal{E}_{\varepsilon,m}$ and $\overline{\mathcal{E}}_{\varepsilon,m}$ are established in term of certain topological invariants of $M$ and $\partial M$ that we quickly recall here, referring for instance to Benci \cite{benci95} for more details on definitions. Given a topological space $X$, denote by $\cat(X)$
its \emph{Lusternik-Schnirelmann category}, i.e. the  
smallest number of sets, open and contractible in $X$,
needed to cover $X$ ($\cat(X) = + \infty$ in case such number does not exist). Denote moreover by $\mathscr{P}_1(X)$ the sum of the \emph{Betti numbers}
of $X$ or, equivalently,
the value at $1$ of the \emph{Poincaré polynomial} of $X$. 

The main results of the paper are the following ones.
\begin{theoremletter}
	\label{theorem:ACH-Neumann}
	Assume that \ref{asu_nondegenerate},
	\ref{asu_coercive} and \ref{asu_subcritical} hold.
	Then, there exists $m^*>0$ such that
        for all $m \in (0,m^*)$
	there exists $\eps_{m},c_{m}>0$ 
	such that for any $\eps \in (0,\eps_{m})$
 the Neumann problem
 \eqref{eq:PDE_Neumann}
	has at least $\mathrm{cat}(\partial M)$ solutions $u_{\eps,m}$
	with $\CE_{\varepsilon}(u_{\eps,m})\le c_{m}$
	and at least one solution $v_{\eps,m}$ with $\CE_{\varepsilon}(v_{\eps,m})> c_{m}$.
	Moreover, if $m$ and $\eps$ as above are such that all critical points of
	$\mathcal{E}_{\varepsilon,m}$ are non-degenerate, then \eqref{eq:PDE_Neumann} has at least $\mathscr{P}_1(\partial M)$
	solutions $u_{\eps,m}$ with $\CE_{\varepsilon}(u_{\eps,m})\le c_{m}$ and at least $\mathscr{P}_1(\partial M)-1$
	solutions $v_{\eps,m}$ with $\CE_{\varepsilon}(v_{\eps,m})> c_{m}$.
\end{theoremletter}
\begin{theoremletter}
	\label{theorem:ACH-Dirichet}
	Assume that \ref{asu_nondegenerate}, \ref{asu_coercive} and \ref{asu_subcritical} hold.
	Then, there exists $m^*>0$ such that for all $m \in (0,m^*)$
	there exist $\eps_{m},c_{m}>0$ 
	such that for any $\eps \in (0,\eps_{m})$
	the Dirichlet problem \eqref{eq:PDE_Dirichlet}
	has at least $\mathrm{cat}( M)$ solutions 
	$u_{\eps,m}$
	with $\CE_{\varepsilon}(u_{\eps,m})\le c_{m}$
	and, if $\cat(M) > 1$,
	there exists at least one solution $v_{\eps,m}$ with $\CE_{\varepsilon}(v_{\eps,m})> c_{m}$.
	Moreover, if $m$ and $\eps$ as above are such that all critical points of
	$\overline{\mathcal{E}}_{\varepsilon,\eps}$ are non-degenerate,
	then \eqref{eq:PDE_Dirichlet} has at least $\mathscr{P}_1( M)$
	solutions $u_{\eps,m}$
	with $\CE_{\varepsilon}(u_{\eps,m})\le c_{m}$ and at least $\mathscr{P}_1( M)-1$
	solutions $v_{\eps,m}$ with $\CE_{\varepsilon}(v_{\eps,m})> c_{m}$.
\end{theoremletter}

In section \ref{sect_generic} we discuss some (essentially known) results which state that solutions of \eqref{eq:PDE_Neumann} are non-degenerate in suitable \textit{generic} situations. Roughly speaking, given the equation \eqref{eq:PDE_Neumann}, one can perform an arbitrary small perturbation of a key parameter (more precisely, either the Riemannian metric or the boundary of the domain) so that the resulting problem is such that all solutions are non-degenerate for a dense set of $\eps$ and $m$. That is, the second parts of Theorems \ref{theorem:ACH-Neumann} and \ref{theorem:ACH-Dirichet}, which are stronger than the first ones, apply in \textit{almost} all cases.

\begin{remark}
	Notice that while the lower bound given in Theorem \ref{theorem:ACH-Dirichet} for the Dirichlet problem \eqref{eq:PDE_Dirichlet} is formulated in term of the topology of the whole of $M$ (as in the case without boundary), the lower bound on the number of solutions given by Theorem \ref{theorem:ACH-Neumann} for the Neumann problem \eqref{eq:PDE_Neumann} depends solely
	on the topology of $\partial M$. This fact is related to the different nature of the limiting isoperimetric-type problem linked to each case in the $m,\eps\to 0$ asymptotics.
\end{remark}

\begin{remark}\label{REMARK_partialM}
It is easy to check that one always has $\mathrm{cat}(\partial M)>1$, as $\partial M$ is a closed and compact manifold.
However, this is not always true for $M$
(consider, for instance, a closed ball in $\R^n$).
\end{remark}

\begin{example}
	Let $M$ be a closed ball in $\mathbb{R}^n$.
	Then we have $\mathrm{cat}(\partial M)=\mathscr{P}_1(\partial M)=2$,
	so that by Theorem \ref{theorem:ACH-Neumann}
	we have at least 3 solutions for the Neumann problem \eqref{eq:PDE_Neumann}
	(in the appropriate parameter regime).
	In particular, we do not obtain
	a better result in the non-degenerate case
	when $M$ is a ball.
	As $\cat(M) = 1 $ and $\mathscr{P}_1(M) = 1$,
	Theorem~\ref{theorem:ACH-Dirichet}
	only guarantees the existence of one solution for
	the Dirichlet problem~\eqref{eq:PDE_Dirichlet}.
\end{example}

\begin{example}
	The situation is different if we take $M$ to be equal to an annulus in $\R^2$. In this case, $\partial M$ is the union of two disjoint circles and this implies that $\mathrm{cat}(\partial M)=\mathscr{P}_1(\partial M)=4$,
	so Theorem \ref{theorem:ACH-Neumann} gives in general
	at least 5 solutions for the Neumann problem \eqref{eq:PDE_Neumann}, and at least 7 solutions in the non-degenerate case.
	Moreover, as $\cat(M) = 2$, Theorem~\ref{theorem:ACH-Dirichet}
	ensures the existence of at least three solutions for~\eqref{eq:PDE_Dirichlet}.
\end{example}

\begin{remark}
	The non-degeneracy assumption \ref{asu_nondegenerate} for the potential $W$ is standard and,
	moreover, generic. The coercivity assumption \ref{asu_coercive} is also
	standard and guarantees that the $\Gamma$-convergence results for the energies
	$(\mathcal{E}_{\varepsilon})_{\eps}$ holds, see, e.g., Fonseca and Tartar \cite{fonseca-tartar}, Owen, Rubinstein and Sternberg \cite{owen-rubinstein-sternberg} (in the Euclidean case).
	Regarding the assumption \ref{asu_subcritical},
	it is of more technical nature but
	also standard in variational problems concerning nonlinear elliptic PDEs or differential geometry.
	It ensures that these problems are compact or,
	more precisely, that both $\mathcal{E}_{\varepsilon,m}$
	and $\overline{\mathcal{E}}_{\varepsilon,m}$
	satisfy the Palais-Smale condition.
	As we prove in Section \ref{sect_generic}, it is possible to drop \ref{asu_subcritical}
	and obtain some partial results.
	However, it is not clear to us
	whether Theorems \ref{theorem:ACH-Neumann} and \ref{theorem:ACH-Dirichet}
	still hold true if \ref{asu_subcritical} is dropped.
	Indeed, the coercivity
	assumption \ref{asu_coercive} could be used to construct suitable energy decreasing projections or truncation-type operators in order to restore strong compactness of sublevel sets of $\mathcal{E}_{\varepsilon}$, but such projections or operators would not respect in general the mass constraint.
	One could also wonder about the validity of our result 
	without the compactness assumption on the manifold $M$.
	In any case, these questions, although interesting, go beyond the
	scope of this paper.
\end{remark}

\begin{remark}
Theorem \ref{theorem:ACH-Neumann} and  Theorem \ref{theorem:ACH-Dirichet} extend to \textit{vectorial double-well potentials} $W\colon \R^k \to \R$ without substantial modifications on the proofs. For simplicity, we have chosen to present the results for scalar potentials. Nevertheless, as noticed in \cite{andrade-conrado-nardulli-piccione-resende} and also as discussed below, the extension to vectorial potentials with three or more phases is not straightforward.
\end{remark}


Our approach to prove Theorem \ref{theorem:ACH-Neumann} and Theorem \ref{theorem:ACH-Dirichet} combines some topological and variational methods with $\Gamma$-convergence properties of $\mathcal{E}_{\varepsilon}$ and tools from Geometric Measure Theory.
Variational methods relating the multiplicity of solutions of an elliptic problem with the topology of the underlying domain were already used in earlier works by Benci, Cerami and Passaseo
\cite{benci-cerami94,benci-cerami91,benci-cerami-passaseo}, see also Benci \cite{benci95} and references therein.
The results contained in this paper are mainly inspired by the recent contribution \cite{benci-nardulli-osorio-piccione},
where the manifold $M$ is supposed to be \emph{closed} and compact (cf. also the earlier result \cite{benci-nardulli-piccione}). 

Roughly speaking,
the idea of the proofs both here and in \cite{benci-nardulli-osorio-piccione}
consists on using the $\Gamma$-convergence properties of the functionals $\mathcal{E}_{\varepsilon,m}$ and $\overline{\mathcal{E}}_{\varepsilon,m}$
in order to show that under small mass constraint and for $\eps>0$ small enough the sublevel sets 
\begin{equation*}
	\mathcal{E}^c_{\eps,m} \coloneqq \left\{ u \in \mathcal{H}_m : \mathcal{E}_{\varepsilon}(u) \leq c  \right\}
	\quad \text{and}\quad
	\overline{\mathcal{E}}^c_{\eps,m} \coloneqq \left\{ u \in \mathcal{H}_{m,0} : \mathcal{E}_{\varepsilon}(u) \leq c  \right\}
\end{equation*}
are, for suitable $c>0$,
homotopically equivalent to $\partial M$ and $M$, respectively.
The topological complexity of the above sublevel sets is hence comparable to that of $\partial M$ and $M$, and we may apply the
infinite-dimensional versions of Lusternik-Schnirelman and Morse theories
(see Palais \cite{palais66,palais63})
to establish the stated lower bounds on the
number of critical points. 

There are other results in related settings
which are also obtained by similar strategies. For instance, Jerrard and Sternberg
\cite{jerrard-sternberg} provided an abstract framework which allows to use
``critical points" (in a non-smooth sense) of a $\Gamma$-limit functional in order to prove existence of critical points for the $\Gamma$-converging functionals when they are
``close" enough to the $\Gamma$-limit. They applied this result to the 2D
Allen-Cahn and 3D Ginzburg-Landau (with and without magnetic field)
functionals. Improvements of this abstract approach in the concrete setting of Ginzburg-Landau functionals on manifolds without boundary have been made recently by Colinet, Jerrard and Sternberg \cite{colinet-jerrard-sternberg} in 3D and by De Philippis and Pigati \cite{dephilippis-pigati} in arbitrary dimension. Similar ideas, but also using different methods, had been applied before by Pacard and Ritoré \cite{pacard-ritore} and Kowalczyk \cite{kowalzcyk2005} in the Allen-Cahn setting. In the same spirit,
Gaspar and Guaraco \cite{gaspar-guaraco,guaraco} proved that the number of solutions of
the Allen-Cahn equation (without volume constraint) in a
closed manifold goes to infinity as $\eps \to 0$. Their proof is restricted to
even potentials, see also Passaseo \cite{passaseo}. 
By similar methods, Bellettini and Wickramaseckera \cite{bellettini-wickramasekera} proved existence of solutions for a more general equation of Allen-Cahn type. Related results for the Ginzburg-Landau equations on manifolds were obtained by Stern \cite{stern21}, and by Pigati and Stern \cite{pigati-stern21} in the case with magnetic field, see also \cite{canevari-dipasquale-orlandi1,canevari-dipasquale-orlandi2}. Moreover, earlier work of this type exists also in the context
of the 2D Ginzburg-Landau functional considered by Bethuel, Brezis and Helein
\cite{bethuel-brezis-helein}. Almeida and Bethuel
\cite{almeida-bethuel97,almeida-bethuel} proved that there exists a lower bound
on the number of critical points according to the topological degree of the
boundary data as long as the perturbation parameter $\eps>0$ is small enough (see also Zhou and Zhou \cite{zhou-zhou} for improvements).

In our case, the $\Gamma$-limit functionals related to $\mathcal{E}_{\varepsilon,m}$ and $\overline{\mathcal{E}}_{\varepsilon,m}$ correspond to different isoperimetric-type problems on $M$ in the small volume regime, for which there are plenty of available results which provide a quite complete picture. In \cite{berard-meyer82}, Bérard and Meyer showed that
the isoperimetric profile of
a $n$-dimensional compact and smooth Riemannian manifold
becomes comparable to the isoperimetric profile of $\R^n$
as the volume tends to zero.
Morgan and Johnson \cite{morgan-johnson} took one step further
by proving that solutions to the isoperimetric problem
with small volume are close to small spheres,
so that their diameter tends to zero for vanishing volumes.
More precisely, a result in \cite{MR2529468} implies
that these spheres are exactly those
centered at a point of maximal scalar curvature
(of the manifold), see also Druet \cite{druet02}.
Analogous results are available
in the setting of compact Riemannian manifolds with boundary (without taking the boundary into account):
Bayle and Rosales \cite{bayle-rosales}
proved the asymptotic equivalence with the
relative isoperimetric profile of the \emph{half}-plane
and Fall \cite{fall10} showed
that relative isoperimetric regions of small volume
are close to half-spheres centered at the
boundary points of maximal mean curvature. Moreover, one might also consider variants of the isoperimetric problem in which the boundary part of the manifold is also taken into account, still obtaining an asymptotic equivalence with the isoperimetric problem in $\R^n$.


One natural question concerns the extension of the results obtained in Theorem \ref{theorem:ACH-Neumann} and Theorem \ref{theorem:ACH-Dirichet} to \textit{systems} of Allen-Cahn-type. It turns out that the main obstacle for such extensions lies in the fine understanding of the resulting $\Gamma$-limit problem. It is known since the work of Baldo \cite{baldo} that the $\Gamma$-limit functional for vectorial multi-well potentials corresponds to a minimization problem for the perimeter of \textit{clusters} (with cardinality corresponding to the number of wells of the potential) under volume constraints. These problems can be seen as vectorial versions of the isoperimetric problem, and they are still poorly understood in the small volume regime. In particular, as pointed out in \cite{andrade-conrado-nardulli-piccione-resende}, in the vectorial case
it is not known in general whether the diameters of the different cluster components of minimizers are uniformly controlled by their respective volume, which is a crucial ingredient in our approach. However, properties of that kind are available for clusters of three sets on the plane, allowing to
extend the results of~\cite{benci-nardulli-osorio-piccione} to the case of triple-well potentials on 2D closed manifolds (see~\cite{andrade-conrado-nardulli-piccione-resende}). 

In a different direction, most of the recent results quoted above concerning the existence of critical points of Allen-Cahn and Ginzburg-Landau functionals (e.g., \cite{andrade-conrado-nardulli-piccione-resende,bellettini-wickramasekera,colinet-jerrard-sternberg,dephilippis-pigati,gaspar-guaraco,guaraco,pigati-stern21,stern21}) are proven in the setting of manifolds without boundary.
It is natural to wonder about their validity when one considers an underlying manifold which has a non-empty boundary as we do in this paper. Moreover, not much is known concerning Ginzburg-Landau functionals with a flux constraint on the vorticity, which are naturally associated to a limiting isoperimetric-type problem in codimension 2.
To our knowledge, the only results in this direction were proven in \cite{bethuel-orlandi-smets04} and by Chiron \cite{chiron05}, where they were applied to problems of existence of traveling waves for the Gross-Pitaevskii equation.
It is possible that these results could be extended to some degree and then used to obtain multiplicity results in the spirit of the present paper.

\tableofcontents
\section{Sketch of the proofs}

In this section we outline the main ingredients used in the proofs of Theorems \ref{theorem:ACH-Neumann} and \ref{theorem:ACH-Dirichet}. We recall that a $C^1$ functional $f$ on a Banach manifold $X$
is said to satisfy the Palais-Smale condition
if every Palais-Smale sequence for $f$ admits a converging subsequence,
where $(x_n)_{n \in \N}$ is a Palais-Smale sequence for $f$ if $(f(x_n))_{n \in \N}$ is bounded and $(\lVert \de f(x_n) \rVert_{X^*})_{n \in \N}$ tends to $0$ as $n \to \infty$. The Palais-Smale condition is a classical and useful sufficient condition
in critical point theory. When it holds one only needs to
detect a change of topology in the level sets of the functional in order to
prove the existence of critical points. The interested reader can find more details in \cite{palais63}.
In our case, the Palais-Smale condition is fulfilled:
\begin{lemma}
	\label{lem:E_eps_C2-PS}
	Assume that \ref{asu_subcritical} holds.
	Then, for all $\eps>0$ and $m \in \R$
	the functional $\mathcal{E}_{\varepsilon}$ is well defined as a
	$C^2$ functional on $H^1(M,\mathbb{R})$
	and it verifies the Palais-Smale condition.
	In particular, $\mathcal{E}_{\varepsilon}$ is a $C^2$ functional
	satisfying the Palais-Smale condition on both $\mathcal{H}_m$
	and $\mathcal{H}_{m,0}$.
	Moreover,
	if $u_{\eps,m}\in \mathcal{H}_m$
	is a critical point of $\mathcal{E}_{\varepsilon,m}$,
	then there exists $\lambda_{\eps,m}$ such that
	$(u_{\eps,m},\lambda_{\eps,m})$
	solves \eqref{eq:PDE_Neumann}
	and 
	if $u_{\eps,m}\in \mathcal{H}_{m,0}$ is a critical point of
	$\overline{\mathcal{E}}_{\varepsilon,m}$,
	then there exists $\lambda_{\eps,m}$
	such that $(u_{\eps,m},\lambda_{\eps,m})$
	solves \eqref{eq:PDE_Dirichlet}.
\end{lemma}
The proof of Lemma \ref{lem:E_eps_C2-PS} is standard. In particular,
compactness of the Palais-Smale sequences for $\mathcal{E}_{\varepsilon}$ follows from the
compactness of $M$ along with the subcritical behavior of the potential
provided by assumption~\ref{asu_subcritical}
(for more details see, for instance, \cite[Proposition 4.12]{benci-nardulli-osorio-piccione}).

According to the previous discussion, the proofs of Theorems~\ref{theorem:ACH-Neumann}
and \ref{theorem:ACH-Dirichet} boil down to determining changes of topology in the sublevel sets of the functionals $\CE_{\eps,m}$ and $\overline{\CE}_{\eps,m}$. To do this, we employ the
so-called \emph{photography method},
which allows to understand some topological invariants of 
a sublevel of a functional through a homotopy equivalence with another topological space for which these invariants are more easily computable.
More precisely, we will use the following result
(see, e.g., \cite{benci95,benci-cerami94}).

\begin{theoremletter}
	\label{theorem:photography}
	Let $X$ be a topological space,
	$\mathfrak{M}$ be a $C^2$-Hilbert manifold,
	$f\colon\mathfrak{M}\to\mathbb{R}$ be a $C^1$-functional.
	For $c\in\mathbb R$, let $f^c\coloneqq\{u\in\mathfrak{M} : f(u)\le c\}$
	be the $c$-sublevel set of $f$.
	Assume that
	\begin{enumerate}
		\item[$($$\mathrm{E1}$$)$]
			$f$ is bounded below;
		\item[$($$\mathrm{E2}$$)$]
			$f$ satisfies the Palais--Smale condition;
		\item[$($$\mathrm{E3}$$)$]
			There exists $c \in \R$ and two continuous maps
			$\Psi_{R}\colon X\rightarrow f^c$ and
			$\Psi_{L}\colon f^c\rightarrow X$
			such that $\Psi_{L}\circ\Psi_{R}$ is homotopic
			to the identity map of $X$.
	\end{enumerate} 
	Then, the number of critical points in $f^c$ 
	is greater than $\cat(X)$. If $\mathfrak{M}$ is contractible and $\cat(X)>1$,
	there is at least another critical point of $f$ outside
	$f^c$.
	Moreover, there exists $c_0\in(c,\infty)$ such that one of
	the two following conditions holds:
	\begin{itemize}
		\item[{\rm (i)}]
			$f^{c_0}$ contains infinitely many critical points;
		\item[{\rm (ii)}]
			$f^{c}$ contains $\mathscr{P}_1(X)$ critical points
			and $f^{c_0}\setminus f^{c}$
			contains $\mathscr{P}_1(X)-1$ critical points
			if counted with their multiplicity.
			More exactly, we have the following relation:
			\begin{equation}
				\label{eq:morserelation}
				\sum_{u\in {\rm Crit}(f)}
                t^{\mu(u)}
                =
                t\mathscr{P}_t(X)
                +t^2[\mathscr{P}_t(X)-1]
                +t(1+t)\mathcal{Q}(t),
			\end{equation}
                where 
                ${\rm Crit}(f)$ is the set of critical points of
			$f$,
                $\mu(u)$ is the Morse index of the solution $u$
			and $\mathcal{Q}(t)$ is a polynomial
			with nonnegative integer coefficients.
			In particular, if all the critical points are non-degenerate,
			there are at least $\mathscr{P}_1(X)$
			critical points in $f^{c_0}$
			and $\mathscr{P}_1(X)-1$ critical points in $f^{c_0}\setminus f^{c}$.
	\end{itemize}
\end{theoremletter}

Since $\mathcal{E}_{\varepsilon}$ is bounded below (as it is nonnegative)
and satisfies the Palais-Smale condition by Lemma \ref{lem:E_eps_C2-PS},
Theorem~\ref{theorem:photography}
implies Theorem~\ref{theorem:ACH-Neumann}
if we are able to show that for all
$m$ and $\eps$ sufficiently small there exist $c > 0$ and two continuous maps
$\mathcal{P}_{\varepsilon,m}\colon \partial M \to \mathcal{E}^c_{\eps,m}$
(the \emph{photography} map)
and 
$\mathcal{B}\colon \mathcal{E}^c_{\eps,m} \to \partial M$
(the \emph{barycenter} map)
such that $\mathcal{B}\circ \mathcal{P}_{\varepsilon,m}
\colon \partial M \to \partial M$
is homotopic to the identity map.
Similarly, Theorem~\ref{theorem:ACH-Dirichet} will follow
by showing the existence of two continuous maps 
$\mathcal{P}_{\varepsilon,m}\colon  M \to \overline{\mathcal{E}}^c_{\eps,m}$
and 
$\mathcal{B}\colon \overline{\mathcal{E}}^c_{\eps,m} \to  M$
such that $\mathcal{B}\circ \mathcal{P}_{\varepsilon,m} \colon M \to M$
is homotopic to the identity map of $M$.
Analogously to \cite{benci-nardulli-osorio-piccione}, the two main ingredients that allow to construct the maps
$\mathcal{P}_{\epsilon,m}$ and $\mathcal{B}$ are the following:
\begin{enumerate}
	\item The $\Gamma$-convergence,
            as $\varepsilon \to 0^+$,
		of the energies $\mathcal{E}_{\varepsilon}$
		towards some form of isoperimetric problem,
		the latter depending one the chosen boundary condition.
		In the case of Neumann boundary conditions,
		the $\Gamma$-limit is the classical relative isoperimetric problem,
		i.e., the boundary of $M$ is not taken into account
		on the computation of the perimeter.
		If one considers homogeneous Dirichlet boundary conditions,
		then the $\Gamma$-limit takes into account
		the perimeter intersecting the boundary $\partial M$.
	\item The fact that the isoperimetric profile and the relative isoperimetric profile on compact manifolds with boundary
		are asymptotic as $m \to 0^+$ to the isoperimetric profiles on the
		Euclidean space and half-space, respectively. 
        As a consequence, for sufficiently small volumes any
        ``almost minimizer'' (a concept that will be formally defined later) of the relative isoperimetric problem
        concentrates almost its entire volume around a point 
        on the boundary of $M$.
		For the isoperimetric problem, the same concentration phenomena
		 occurs but for points on $M$. See Figure \ref{FIG_am} for an explanatory picture.
  \end{enumerate}
		Combining this two ingredients, we find that in the Neumann case
		any function that belongs to the
		smallest sublevel
		$\mathcal{E}_{\varepsilon,m}^c$
		that contains the image of the photography map
		has almost all its mass
		in a geodesic ball centered in a point of $\partial M$.
		By composing with the nearest point projection on $\partial M$, it follows that the barycenter map
		$\mathcal{B}\colon \mathcal{E}_{\varepsilon,m}^c \to \partial M$
		is well defined and continuous.
		In the Dirichlet case, we have that an almost minimizer
		of $\overline{\mathcal{E}}_{\epsilon,m}$
		has almost all its mass in a geodesic ball centered in a point of $M$,
		so that the barycenter map 
		$\mathcal{B}\colon \mathcal{E}_\varepsilon^c \to  M$
		is well defined and continuous. Conversely, one finds the photography maps $\CP_{\eps,m}$ by constructing almost minimizers of $\CE_{\eps}$ supported on geodesic balls centered at arbitrary points of $\partial M$ in the Neumann case and of $M$ in the Dirichlet case.

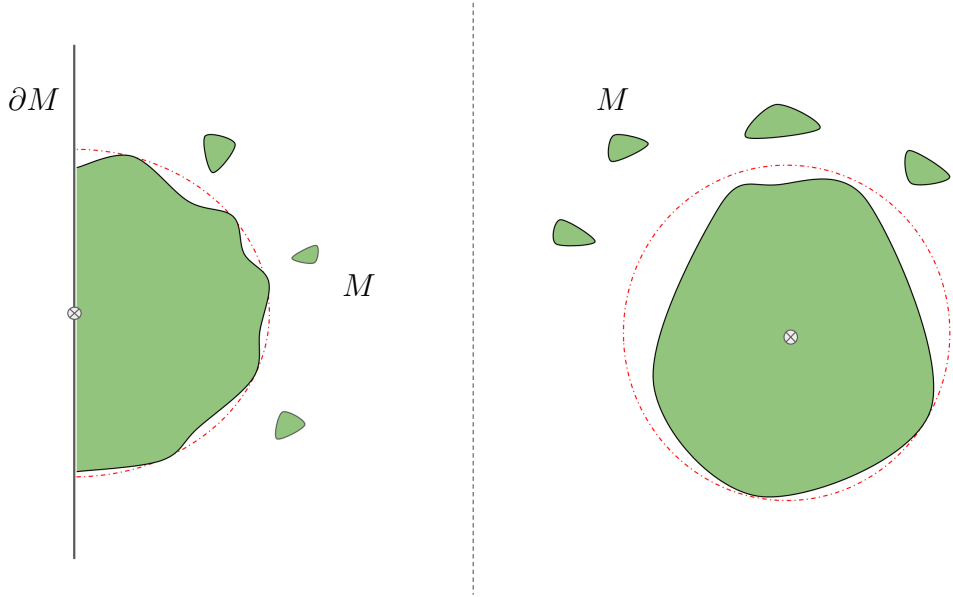
\begin{figure}
    \centering
    \begin{picture}(400,200)
        \put(0,0){\includegraphics[width=\textwidth]{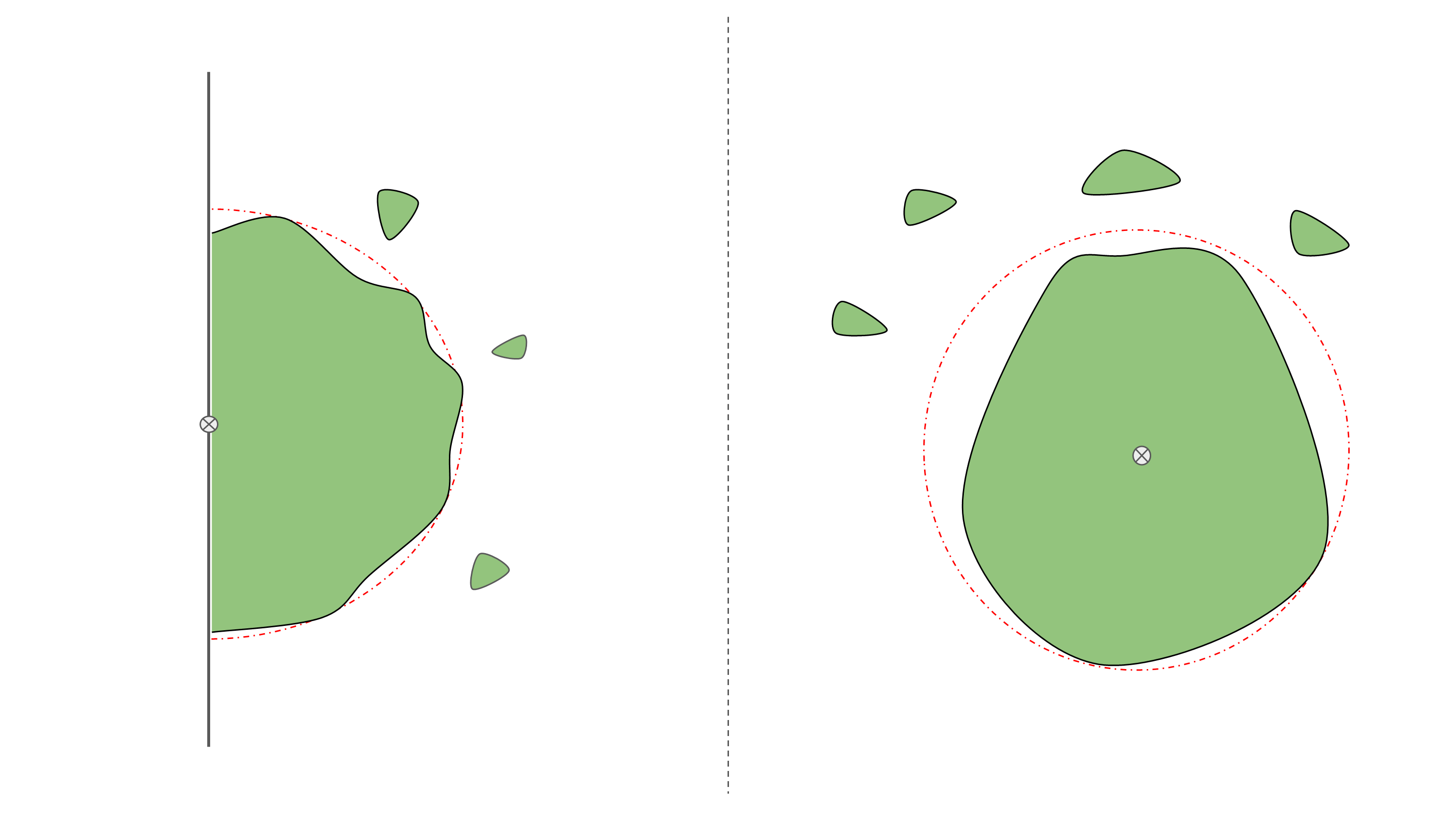}}
        \put(35,190){$\partial M$}
        \put(160,120){$M$}
        \put(255,190){$M$}
    \end{picture}
    \caption{The filled regions represent an arbitrary almost minimizer of the relative isoperimetric problem (left) and the isoperimetric problem (right). In the former case, most of the volume is concentrated inside a ball (dotted line) 
    a half-ball centered at the boundary.
    In the latter case, concentration occurs around some point of $M$. However, some of the volume might be away from these balls, but it is always a small portion of the volume which can be controlled. In particular, half-balls and balls  are almost minimizers of the relative isoperimetric and isoperimetric problems, repectively.}
    \label{FIG_am}
\end{figure}

\begin{remark}
	\label{rem:cat-partial-M}
	To apply Theorem~\ref{theorem:photography}
	in the proof of Theorem~\ref{theorem:ACH-Neumann},
	one needs that $\cat(\partial M)>1$
	and that $\mathcal{H}_{m}$ is contractible.
	Since these two conditions are always satisfied (in fact, $\mathcal{H}_{m}$ is convex),
	they are omitted in the statement of Theorem~\ref{theorem:ACH-Neumann}. This is also the reason why the condition $\mathrm{cat}(M)>1$, not necessarily fulfilled a priori, appears in the statement of Theorem \ref{theorem:ACH-Dirichet}.
\end{remark}

\section{Proof of Theorem~\ref{theorem:ACH-Neumann}}

\subsection{The relative isoperimetric problem}
\label{sect_iso}
We recall some standard notations:
let $BV(M,\R)$ denotes the set of functions of bounded variation defined on $M$;
for any $N \subset M$, we denote by $\mathbf{1}_N$ its characteristic function;
$\CH^{n}$ the $n$--dimensional Hausdorff measure on
$\R^{\tilde{n}}$; for any measurable set $A \subset M$,
we denote by $\partial^*A$ the reduced boundary relative to $M$,
which is the subset of $\partial A$ whose notion of 
measure-theoretic normal vector exists 
and has length equal to one.
We denote by $\mathcal{C}_g(M)$ the set of Caccioppoli subsets of $M$,
hence 
\[
	\mathcal{C}_g(M)\coloneqq
	\Big\{
		\Omega\subset M:\,
		\Omega \text{ is measurable and }
		\mathcal{H}^{n-1}(\partial^*\Omega) < +\infty
	\Big\}.
\]
Moreover, let $\sigma > 0$ be the following constant, which depends only on $W$:
\begin{equation}
	\label{eq:def-sigma}
	\sigma \coloneqq 2\, \inf\left\{\int_{0}^{1}\sqrt{W(\gamma(s))}|\gamma'(s)|\de s:
	\gamma \in C^1([0,1],\mathbb{R}), \gamma(0) = 0, \gamma(1) = 1\right\}.
\end{equation}
For $u \in L^1(M,\R)$ we set
\begin{equation*}
	\mathcal{E}_0(u)\coloneqq  \begin{cases}
		\sigma\CH^{n-1}(\partial^*\Omega \cap \mathrm{int}(M)),
		&\mbox{ if } u=\mathbf{1}_{\Omega},\\
		+\infty, &\mbox{ otherwise}.
	\end{cases}
\end{equation*}

Alternatively, one can define $\mathcal{E}_0$ as
the relative perimeter of $\Omega$,
which is defined as the total measure of
$\lvert \nabla\mathbf{1}_{\Omega} \rvert$.
The fact $\mathbf{1}_{\Omega} \in BV(M,\R)$
implies that both definitions are equivalent,
see Giusti \cite{giusti84}
(see also Sternberg \cite{sternberg88} for more details).

For fixed $m \in [0,\mathrm{vol}(M)]$, we define 
\begin{equation}
	\label{IP_regular}
	I_{M}(m)\coloneqq 
	\inf\left\{ \mathcal{H}^{n-1}(\partial^*\Omega\cap\mathrm{int}(M)):
		\Omega \in \mathcal{C}_g(M)
		\text{ and }
	\int_M \mathbf{1}_{\Omega} \de v_g = m\right\}.
\end{equation}
Notice that $I_{M}\colon [0,\mathrm{vol}(M)]\to \mathbb{R}$
is the \emph{isoperimetric profile function} of $M$
and that 
\begin{equation}
	\label{eq:I_M-infE_0}
	I_M(m) = \frac{1}{\sigma} \inf\left\{
		\mathcal{E}_0(u):
		u \in L^1(M,\mathbb{R}),\,
		\int_{M}u\, \de v_g = m
	\right\}.
\end{equation}
The same problem can be stated in the euclidean semi-space
\[
	\mathbb{R}^n_+=\left\{ z \in \mathbb{R}^n: z^n \ge 0 \right\}
	= \mathbb{R}^{n-1}\times [0,+\infty[,
\]
so that the function
$I_{\mathbb{R}^n_+}\colon (0,+\infty) \to \mathbb{R}$
is defined analogously
and there exists a dimensional constant $c_n^+ > 0$
such that 
\begin{equation}
	\label{eq:def-cn}
	I_{R^{n}_{+}}(m) = c_n^+ m^{\frac{n-1}{n}}.
\end{equation}
When $m$ is small, there exists a link between 
the isoperimetric function $I_M$
and its euclidean counterpart.
More formally, Bayle and Rosales \cite{bayle-rosales}
(see also Fall~\cite{fall10})
have shown that
\begin{equation}
	\label{eq:IM-to-IR+}
	I_M(m) = \big(1 + \mathcal{O}(m^{1/n})\big)I_{\mathbb{R}^n_+}(m).
\end{equation}

\subsection{$\Gamma$-convergence}
\label{sect_GC}
We now state the two main $\Gamma$-convergence results we need,
starting with the the most celebrated one,
first proven by Modica and Mortola \cite{modica,modica-mortola}, and later extended to the vector-valued setting by several authors:
Baldo \cite{baldo},
Fonseca and Tartar \cite{fonseca-tartar} and
Sternberg \cite{sternberg88}.
In \cite{fonseca-tartar}, the double-well case was considered
without the subcritical growth condition \ref{asu_subcritical},
meaning that their result is the most well adapted to our situation.
We also point out that in the above mentioned papers the results
are proven in the Euclidean case,
but it was later shown
(see Benci et al. \cite{benci-nardulli-osorio-piccione}
and the references therein)
that they can be extended to the Riemannian setting.
\begin{theorem}[cf. \cite{fonseca-tartar}, Theorem 3.4 and \cite{benci-nardulli-osorio-piccione}, Proposition 3.3]
	\label{theorem:G-convergence}
	Assume that \ref{asu_nondegenerate} and \ref{asu_coercive} hold.
	Then, the following statements hold:
	\begin{enumerate}[i)]
		\item \text{\rm ($\Gamma$-$\liminf$ inequality)}:
			If $(\eps_k)_{k \in \N}\subset ]0,+\infty[$
			is such that $\eps_k \to 0^+$
			and $(u_{\eps_k})_{k \in \N} \subset H^1(M,\R)$
			is such that $u_{\eps_k} \to u_0$ in $L^1(M,\R)$,
			then $\liminf_{k \to \infty}
			\mathcal{E}_{\varepsilon_k}(u_{\eps_k}) \ge \mathcal{E}_0(u_0)$;
		\item \text{\rm ($\Gamma$-$\limsup$ (in)equality)}:
			For any $u_0 \in L^1(M,\R)$ such that $u_0 = \mathbf{1}_{\Omega}$
			for some $\Omega \subset M$ and for every sequence
			$(\eps_k)_{k \in \N}\subset ]0,+\infty[$ such that $\eps_k \to 0^+$,
			there exists $(u_{\eps_k})_{k \in \N} \subset H^1(M,\R)$
			such that $u_{\eps_k} \to u_0$ in $L^1(M,\R)$,
			\[
				\int_{M} u_{\varepsilon_k} \de v_g = \int_M u_0\de v_g
				\quad\text{and}\quad
				\limsup_{k \to \infty}
				\mathcal{E}_{\varepsilon_k}(u_{\eps_k}) \le \mathcal{E}_0(u_0).
			\]
	\end{enumerate}
\end{theorem}

\begin{remark}
	\label{rem:epsilon-approximation}
	The interested reader can find
	in~\cite[Proposition 3.3]{benci-nardulli-osorio-piccione}
	a detailed and explicit construction
	of the functions $u_{\varepsilon}$
	that approximate $u_0$ as required by
	the $\Gamma$-$\limsup$ statement of Theorem~\ref{theorem:G-convergence},
	which is the Riemannian version of the construction
	given by Modica in~\cite{modica} in the Euclidean setting.
	We give here just the idea of this construction, following \cite{modica}.
	For every $A \subset M$ we denote by
	$d_A\colon M \to \mathbb{R}$ the following function:
	\begin{equation}
		\label{eq:def-d_A}
		d_{A}(x) \coloneqq
		\begin{dcases}
			-\dist_g(x,\partial A), &		\mbox{if $x\in A$} , \\
			\phantom{-}\dist_g(x,\partial A), &	\mbox{if $x \notin A$}.
		\end{dcases}
	\end{equation}
	For every $\varepsilon > 0$,
	let us consider an increasing and Lipschitz continuous 
	function $\widetilde{q}_{\varepsilon}\colon \mathbb{R} \to [0,1]$
	such that 
	$\widetilde{q}_{\varepsilon}(t) = 0$ if $t < 0$,
	and $\widetilde{q}_{\varepsilon}(t) = 1$
	if $t \ge \eta_{\varepsilon} > 0$,
	where $\eta_{\varepsilon}\to 0$ as $\varepsilon\to 0$.
	We remind that such a function is linked with the
	energy functional $\mathcal{E}_{\varepsilon}$ since it is a solution of the 
	following differential equation:
	\begin{equation}
		\label{eq:ODE-tilde-q}
		\frac{\de}{\de t}
		\widetilde{q}_{\varepsilon}
		= \frac{1}{\varepsilon}
		\sqrt{\varepsilon^{3/2} + 2W(\widetilde{q}_{\varepsilon})}.
	\end{equation}
	Finally, for every $u_0 = \mathbf{1}_{\Omega}$,
	$u_{\varepsilon}$ is obtained as
	\begin{equation}
		\label{eq:modica-approximation-qtilde}
		u_{\varepsilon}(x) 
		= \widetilde{q}_{\varepsilon}(d_{M \setminus \Omega}(x)
		+ \delta_{\varepsilon,\Omega}),
	\end{equation}
	where $\delta_{\varepsilon,\Omega}\in [0,\eta_{\varepsilon}]$
	is a correction term such that
	\begin{equation}
		\label{eq:def-delta-epsilon-Omega}
		\int_{M}u_{\varepsilon}\, \de v_g
		= \int_{M}
		\widetilde{q}_{\varepsilon}(d_{M \setminus \Omega}(x)
		+ \delta_{\varepsilon,\Omega}) \de v_g
		= 
		\int_{M} u_0 \, \de v_g.
	\end{equation}
\end{remark}

\begin{theorem}[cf. Theorem 4.1 of \cite{fonseca-tartar}]
	\label{theorem:recovery-sequence}
	Assume that \ref{asu_nondegenerate} and \ref{asu_coercive} hold.
	If $(\varepsilon_k)_k \subset \mathbb{R}$
	is a sequence of positive numbers 
	such that $\varepsilon_k \to 0^+$
	and $(u_{\varepsilon_k})_k \subset H^1(M,\mathbb{R})$
	is a sequence of functions such that
	\begin{equation}
		\label{eq:uk-bounded}
		\mathcal{E}_{\varepsilon_k}(u_{\varepsilon_k}) \le c,
		\qquad \forall k \in \mathbb{N},
	\end{equation}
	for some constant $c > 0$,
	then there exists a subsequence,
	still denoted by $(\varepsilon_k)_{k}$,
	such that $(u_{\varepsilon_k})_k$
	converges to a function $u_0 \in L^1(M,\R)$
	with respect to the $L^1$--norm.
\end{theorem}

\begin{remark}
	\label{rem:convergence-valueE0}
	As a direct consequence of Theorem~\ref{theorem:recovery-sequence}
	and the lim-inf property of Theorem~\ref{theorem:G-convergence},
	if $\mathcal{E}_{\epsilon_k}(u_{\varepsilon_k}) \le c$
	for every $k \in \mathbb{N}$,
	then the function $u_0\in L^1(M,\R)$
	such that $u_{\varepsilon_k}\to u_0$
	up to subsequences is such that
	$\mathcal{E}_{0}(u_0) \le E^*$,
	there exists a measurable set $\Omega \subset M$
	such that $u_0 = \mathbf{1}_{\Omega}$
	and 
	\begin{equation*}
		\lim_{k \to \infty} \int_M u_k\, \de v_g
		= \int_M u_0\, \de v_g.
	\end{equation*}
\end{remark}

\subsection{Photography map}
\label{sect_photography}

In order to obtain the map $\mathcal{P}_{\varepsilon,m}$,
we combine the $\Gamma$-convergence result
given by Theorem~\ref{theorem:G-convergence} with
the characterization of the isoperimetric regions 
in a manifold with boundary given by Fall~\cite{fall10}.
Roughly, as long as $m$ and $\varepsilon$ are sufficiently small,
for every point $p \in \partial M$ we define 
$\mathcal{P}_{\varepsilon,m}(p)$
as the $\varepsilon$--approximation ensured by 
the lim-sup part of Theorem~\ref{theorem:G-convergence}
of the characteristic function of a small perturbation
of a half-ball centered at $p$ and with volume $m$,
denoted by $E_{p,m}\subset M$.
Such a set is defined by employing the normal coordinates at the point $p$,
and by using the implicit function theorem to construct the perturbation
that minimizes the perimeter.
The details of the construction of $E_{p,m}$ can be explored
in~\cite{fall10}.
However, for the purposes of this article,
we will provide a concise overview of the essential details.

Let $N_{\partial M}$ the unit interior normal vector field
along $\partial M$ and, 
for any $p \in \partial M$,
let $\mathrm{exp}^{\partial M}_{p}$ be the exponential map
of $\partial M$ at $p$.
Let us consider an orthonormal frame field
$(e_1,\dots,e_{n-1},N_{\partial M})$ of $M$ along $\partial M$.
We define the map $f^p\colon \mathbb{R}^{n-1} \to \partial M$
as follows:
\begin{equation}
	\label{eq:def-fp}
	f^p(z)\coloneqq \mathrm{exp}^{\partial M}_{p}(z^{i}e_{i}),
\end{equation}
where we used the Einstein summation convention.
For every $z \in \mathbb{R}^{n-1}$ and $t \in \mathbb{R}$,
with $t \ge 0$ and sufficiently small,
let us define also
\begin{equation}
	\label{eq:def-Fp}
	F^{p}(z,t) = \exp_{f^p(z)}(t N_{\partial M}) \in M,
\end{equation}
so that $F^p$ provides a local parametrization of a neighborhood 
of $p \in M$.
Our aim is to define a ``semi-sphere'' on $M$ with center in $p$.
To this aim, let us denote by $B^{n-1}$ the unit ball in $\mathbb{R}^{n-1}$
and by $S^{n-1}_+\in \mathbb{R}^n$ the upper hemisphere of
$S^{n-1}$, hence
\[
	S^{n-1}_+ \coloneqq
	\big\{z \in \mathbb{R}^{n}: \norm{z} = 1, z^n \ge 0\big\}.
\]
Let us denote by $\Theta\colon B^{n-1} \to S^{n-1}_+$
the inverse of the stereographic projection from the 
south pole $(0,\dots,0,-1) \in \mathbb{R}^n$,
formally seeing $B^{n-1}$
as $\{z \in \mathbb{R}^n: \norm{z}\le 1, z^n = 0\}$.		
Let $\widetilde{\Theta}\colon B^{n-1} \to B^{n-1}$
and $t\colon B^{n-1} \to [0,1]$ be such that
\[
	\Theta(z) = \big(\widetilde{\Theta}(z),t(z)\big).
\]
For every $r$ sufficiently small, the hypersurface
defined by 
\[
	\left\{
		F^p\big( r\widetilde{\Theta}(z), r t(z)\big):
		z \in B^{n-1}
	\right\} \subset M
\]
can be seen as an semi-sphere centered at $p$.
All the hypersurfaces nearby this semi-sphere centered at $p\in\partial M$
can be obtained by applying a small perturbation
$\omega\colon  S^{n-1}_+\to \mathbb{R}$,
and so we define the hypersurface 
$\Sigma_{p,r,\omega} \subset M$ as follows:
\begin{equation}
	\label{eq:def-Sigma-p-r-omega}
	\Sigma_{p,r,\omega}
	\coloneqq
	\left\{
		F^p\Big( \big(r + \omega(\Theta(z))\big)\, \widetilde{\Theta}(z),
		\big(r + \omega(\Theta(z))\big) t(z)\Big):
		z \in B^{n-1}
	\right\}.
\end{equation}
Since $t(z) = 0$ for every $z \in \partial B^{n-1}$,
$\partial \Sigma_{p,r,\omega} \subset \partial M$,
so that $\Sigma_{p,r,\omega}$ and $\partial M$
enclose a set that we denote by 
$E_{p,r,\omega}$.

By~\cite[Lemma 4.6] {fall10},
there exists $r_0 > 0$ 
such that for any $p \in \partial M$
and $r \in (0,r_0)$
there exists a unique smooth
$\omega^{p,r} \in C^{2,\alpha}(S^{n-1}_+)$
such that 
$\Sigma_{p,r,\omega^{p,r}}$
has a mean curvature function
that is, modulo an additive constant, 
an eigenfunction
of the first strictly positive eigenvalue
of the Laplacian on the hemisphere.
In this way,
$\Sigma_{p,r,\omega^{p,r}}$ represents an almost CMC
(constant mean curvature)
half-sphere centered at $p$,
thereby implying that it is the unique competitor
among all hypersurfaces $\Sigma_{p,r,\omega}$
to serve as a solution to the isoperimetric problem.
Moreover, $\norm{\omega^{p,r}}_{_C^{1,\alpha}(S^{n-1}_{+})}\to 0$
as $r \to 0$ and, since it is defined by using the 
Implicit Function Theorem, the map 
$(p,r)\mapsto \omega^{p,r}$ is continuous.
We denote as \emph{pseudo half-bubbles}
the hypersurfaces $\Sigma_{p,r,\omega^{p,r}}$,
in analogy with the \emph{pseudo-bubbles} introduced in~\cite{MR2529468}.

Following the notation in~\cite{fall10},
we will use the simpler notation $E_{p,r}$
to refer to $E_{p,r,\omega^{p,r}}$.
Furthermore, for any $p \in \partial M$
and sufficiently small $m$,
we denote by $r_{p,m} > 0$ the unique radius such that
$E_{p,r_{p,m}}$ has volume $m$,
i.e., $\int_{M}\mathbf{1}_{E_{p,r_{p,m}}} \de v_g = m$.
For more detailed information on the existence of such $r_{p,m}$,
please refer to~\cite[Lemma 4.7]{fall10}.

\begin{remark}
	\label{rem:def-m0}
	By the compactness assumption of $M$
	and $\partial M$,
	there exists $m_0> 0$
	such that for every $p \in \partial M$
	and $m \in (0,m_0)$ we have that 
	$r_{p,m}$ is well defined and less than $r_0$.
	As a consequence, $\omega^{p,r_{p,m}}$
	is always well defined by~\cite[Lemma 4.6]{fall10}
	for every $p \in \partial M$ and $m \in (0,m_0)$.
\end{remark}
For the sake of notation,
for every $p \in \partial M$ and $m \in (0,m_0)$
we set
\[
	E_{p,m} \coloneqq E_{p,r_{p,m}}
	\quad\text{and}\quad
	\Sigma_{p,m}\coloneqq
	\Sigma_{p,r_{p,m},\omega^{p,r_{p,m}}} = 
	\partial E_{p,m} \cap M.
\]
Let us notice that, since $m < m_0$,
the hypersurface $\Sigma_{p,m}$ is smooth.
Finally, we will denote by 
$u^{m}_{0,p} \in L^1(M,\mathbb{R})$ the characteristic function 
of $E_{p,m}$,
hence
\[
	u^{m}_{0,p}(x) \coloneqq
	\begin{dcases}
		1, &		\mbox{if } x \in E_{p,m}, \\
		0, &		\mbox{otherwise}.
	\end{dcases}
\]

Before defining the photography map,
let us state~\cite[Lemma 4.8]{fall10} in accordance with our notation.
\begin{lemma}
	\label{lemma:fall10-418}
	For every $p\in \partial M$,
	let us denote by $H_{\partial M}(p)$
	the mean curvature of $\partial M$ at $p$.
	Then, there exists a dimensional constant
	$\gamma_n > 0$ such that 
	\begin{equation}
		\label{eq:fall-local-convergence}
		\mathcal{E}_{0}(u^m_{0,p})
		= \sigma
		\left(
			I_{\mathbb{R}^n_+}(m)
			- \gamma_n H_{\partial M}(p) m
			+ \mathcal{O}(m^\frac{n+1}{n})
		\right),
		\qquad \forall m \in (0,m_0),
	\end{equation}
	and, moreover, the following holds:
	\begin{equation}
		\label{eq:fall-IM-convergence}
		I_M(m)
		= 
		I_{\mathbb{R}^n_+}(m)
		- \gamma_n  \max_{p \in \partial M}\{H_{\partial M}(p)\} m
		+ \mathcal{O}(m^\frac{n+1}{n}).
	\end{equation}
\end{lemma}

\begin{remark}
	For our specific purposes,
	it is possible to avoid the use of the perturbation $\omega^{p,r}$
	by setting $\omega \equiv 0$
	in the definitions of
	$\Sigma_{p,r,\omega}$, $E_{p,r}$, and $u^m_{0,p}$.
	This choice is viable
	because the primary tool we will employ in the subsequent construction
	is the estimate~\eqref{eq:fall-local-convergence},
	which remains valid even when $\omega \equiv 0$.
	However, in order to avoid a detailed proof of this last assertion,
	as it falls outside the scope of our work,
	we rely on the estimate provided by~\cite{fall10}.
\end{remark}

\begin{definition}
	\label{def:photography}
	For every $m \in (0,m_0)$
	and $\varepsilon > 0$,
	$\mathcal{P}_{\varepsilon,m}\colon \partial M \to \mathcal{H}_{m}$
	is the map that links
	$p \in \partial M$
	to $u^m_{\varepsilon,p}$,
	that is the $\varepsilon$--approximation of 
	$u^m_{0,p} \in L^1(M,\mathbb{R})$
	given by the \emph{lim-sup} property of Theorem~\ref{theorem:G-convergence}
	(see Remark~\ref{rem:epsilon-approximation}).
\end{definition}
\begin{figure}
    \centering
    \begin{picture}(400,200)
        \put(0,0){\includegraphics[width=\textwidth]{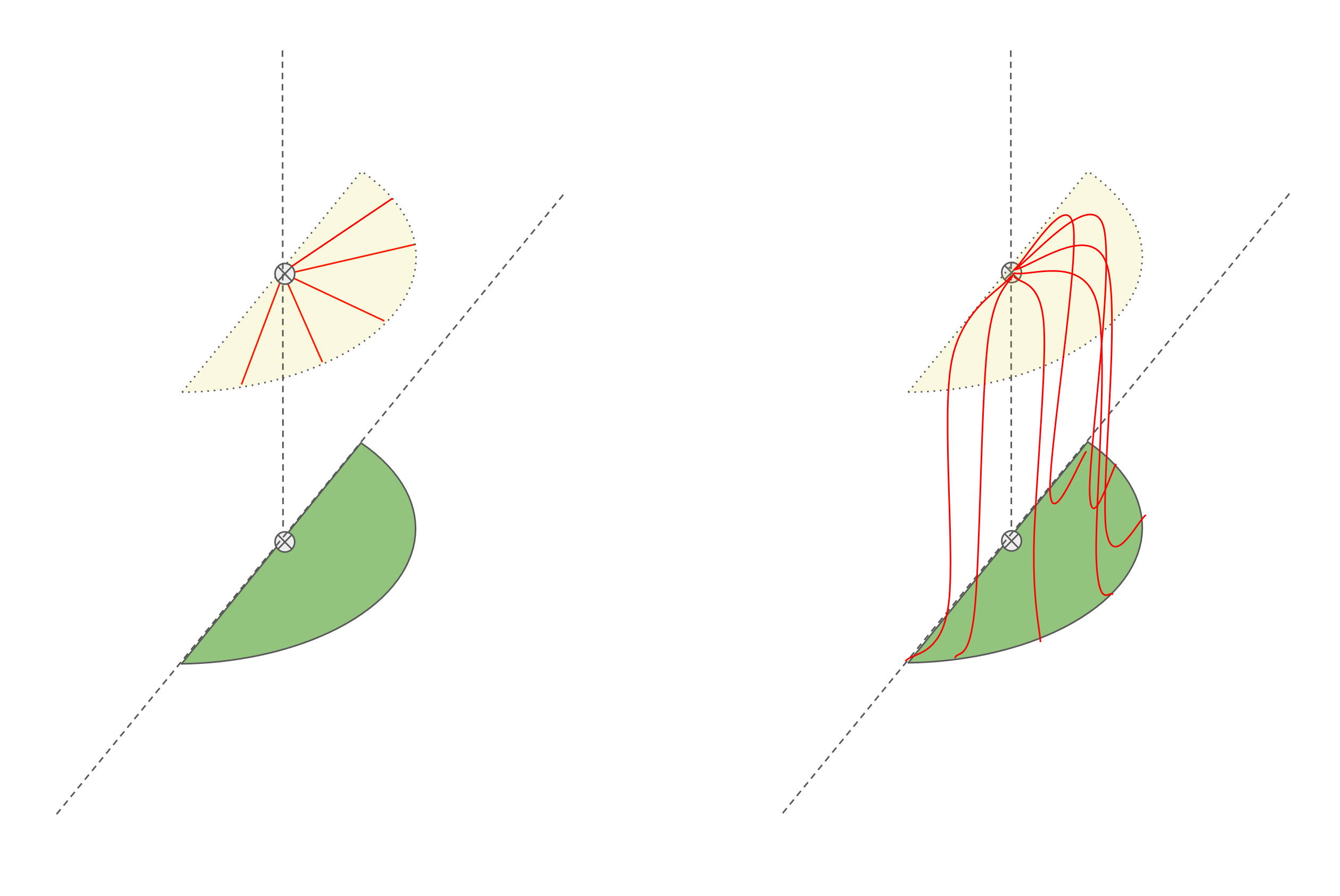}}
        \put(165,190){$\partial M$}
        \put(393,190){$\partial M$}
        \put(135,130){$M$}
        \put(363,130){$M$}
        \put(95,100){$p$}
        \put(323,100){$p$}
        \put(55,90){$0$}
        \put(283,90){$0$}
        \put(55,175){$1$}
        \put(283,175){$1$}
        \put(130,220){$u_{0,p}^m$}
        \put(360,220){$u_{\varepsilon,p}^m$}
    \end{picture}
    \caption{Depiction of the map $\mathcal{P}_{\varepsilon,m}$ introduced in Definition \ref{def:photography}. On the left, we see $u_{0,p}^m$, the indicator function of the set $E_{p,m}$, which is the set enclosed by the semi-sphere $\Sigma_{p,m}$ and $\partial M$. The map $\mathcal{P}_{\varepsilon,m}$ returns $u^m_{\eps,m}$, depicted on the right, which is the smooth approximation of $u_{0,p}^m$ given by the $\Gamma$-convergence result (Theorem \ref{theorem:G-convergence}).}
    \label{FIG_photography_boundary}
\end{figure}
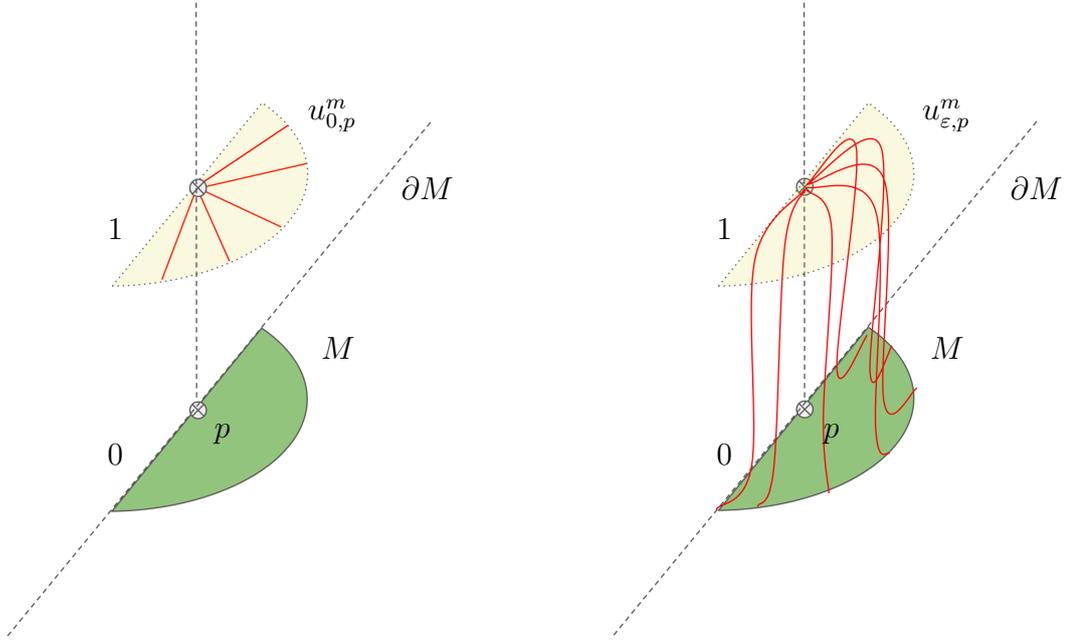

See Figure \ref{FIG_photography_boundary} for a picture aimed at clarifying Definition \ref{def:photography}. The main result of this section is the following,
which provides an estimate on the smallest sublevel of
$\mathcal{E}_{\varepsilon,m}$
that contains the image of the photography map.
\begin{proposition}
	\label{prop:photography-sublevel}
	Assume that \ref{asu_nondegenerate} and \ref{asu_coercive} hold. There exists a constant $\theta = \theta(M,g,W) > 0 $
	such that 
	there exists $m_1 = m_1(M,g,W,\theta) \in ]0,m_0[$
	such that for every $m \in ]0,m_1[$ 
	there exists $\varepsilon_1 = \varepsilon_1(M,g,W,\theta,m) > 0$
	such that for every $\varepsilon \in ]0,\varepsilon_1[$
	we have
	\begin{equation}
		\label{eq:photography-sublevel}
		\mathcal{E}_{\varepsilon,m}(\mathcal{P}_{\varepsilon,m}(p))
		\le \sigma I_M(m) + \theta m,
		\qquad \forall p \in \partial M.
	\end{equation}
	In other words, the sublevel
	$\mathcal{E}_{\varepsilon,m}^{\sigma I_M(m) + \theta m}$
	contains the whole image of the photography map $\mathcal{P}_{\epsilon,m}$.
\end{proposition}

\begin{proof}
	Combining~\eqref{eq:fall-local-convergence}
	and~\eqref{eq:fall-IM-convergence},
	we get 
	\begin{equation}
		\label{eq:photo-sublevel-proof1}
		\mathcal{E}_0(u^m_{0,p})
		= 
		\sigma
		\left(
			I_M(m)
			+ \gamma_n\big(\max_{p \in \partial M}\{H_{\partial M}(p)\}
				- H_{\partial M}(p)
		\big)m\right)
		+ \mathcal{O}(m^{\frac{n+1}{n}}).
	\end{equation}
	Hence, setting $\theta = \theta(M,g) >0$ as
	\begin{equation}
		\label{eq:def-omega}
		\theta = 
		\sigma
		\gamma_n
		\left( \max_{p \in \partial M}\{H_{\partial M}(p)\}
			- \min_{p \in \partial M}\{H_{\partial M}(p)\}
		+ 1 \right),
	\end{equation}
	we have
	\begin{equation}
		\label{eq:photo-sublevel-proof2}
		\mathcal{E}_0(u^{m}_{0,p})
		< \sigma I_{M}(m) + \theta m + \mathcal{O}(m^{\frac{n+1}{n}}),
		\qquad \forall p \in \partial M.
	\end{equation}
	As a consequence, there exists $m_1 = 
	m_1(M,g,\theta) \in (0,m_0)$
	such that for every $m \in (0,m_1)$
	we have
	\begin{equation}
		\label{eq:photo-sublevel-proof3}
		\mathcal{E}_0(u^{m}_{0,p})
		< \sigma I_{M}(m) + \theta m,
		\qquad \forall p \in \partial M.
	\end{equation}
	By Theorem~\ref{theorem:G-convergence}
	and Remark~\ref{rem:epsilon-approximation},
	we also know that 
	\begin{equation}
		\label{eq:photo-sublevel-proof4}
		\limsup_{\varepsilon \to 0^+}
		\mathcal{E}_{\varepsilon}(u^{m}_{\varepsilon,p})
		= 
		\limsup_{\varepsilon \to 0^+}
		\mathcal{E}_{\varepsilon}(\mathcal{P}_{\varepsilon,m}(p))
		\le 
		\mathcal{E}_0(u^{m}_{0,p}),
		\qquad \forall p \in \partial M.
	\end{equation}
	By the strict inequality in~\eqref{eq:photo-sublevel-proof3}
	and since $M$ is a compact manifold,
	by~\eqref{eq:photo-sublevel-proof4}
	we infer the existence of 
	$\varepsilon_1 = \varepsilon_1(M,g,W,\theta,m) > 0$
	such that~\eqref{eq:photography-sublevel} holds
	for every $\varepsilon \in (0,\varepsilon_1)$.
\end{proof}

\begin{remark}
	By~\eqref{eq:def-cn} and~\eqref{eq:IM-to-IR+},
	namely that $I_M(m) \approx m^{\frac{n-1}{n}}$
	as $m$ goes to zero,
	we have 
	\begin{equation}
		\label{eq:omega-m-oIm}
		\lim_{m \to 0^+} \frac{m}{I_M(m)} = 0,
	\end{equation}
	hence the ``width'' of the sublevel 
	that contains the photography goes rapidly to $0$
	as $m \to 0^+$.
\end{remark}

In order to use $\mathcal{P}_{\varepsilon,m}$
as a photography map
in the sense of Theorem~\ref{theorem:photography},
it remains to prove its continuity.
\begin{proposition}
	\label{prop:photography-continuity}
	Assume that \ref{asu_nondegenerate} and \ref{asu_coercive} hold. Let $m_1 > 0$ and $\varepsilon_1 > 0$
	be defined as in Proposition~\ref{prop:photography-sublevel}.
	For every $m \in (0,m_1)$
	and $\varepsilon \in (0,\varepsilon_1)$,
	the map
	$\mathcal{P}_{\varepsilon,m}\colon \partial M \to \mathcal{H}_m$
	is continuous.
\end{proposition}
\begin{proof}
	The proof follows the same line of the one of
	\cite[Proposition 4.14]{benci-nardulli-osorio-piccione},
	and we report here the main steps.

	We recall that for every $p \in \partial M$ and $\varepsilon > 0$
	the function $\mathcal{P}_{\varepsilon,m}(p)$
	is defined by using~\eqref{eq:modica-approximation-qtilde},
	hence we have
	\[
		\big(\mathcal{P}_{\varepsilon,m}(p)\big)(x)
		= 
		\widetilde{q}_{\varepsilon}\big(d_{M \setminus E_{p,m}}(x) +
		\delta_{\varepsilon,E_{p,m}}\big),
	\]
	where $d_{M\setminus E_{p,m}}$ is defined by using~\eqref{eq:def-d_A}
	and
	$\widetilde{q}_{\varepsilon}$ is the Lipschitz continuous 
	function that solves~\eqref{eq:ODE-tilde-q}.
	As a consequence, there exists a constant $C > 0$,
	which depends on $M,g,\varepsilon$ and $W$,
	such that the following inequality holds:
	\begin{multline*}
		\norm*{\mathcal{P}_{\varepsilon,m}(p_1)
		- \mathcal{P}_{\varepsilon,m}(p_2)}_{H^1(M,\mathbb{R})}\\
		\le C\Big[
			\norm{d_{M\setminus E_{p_1,m}}
			- d_{M\setminus E_{p_2,m}}}_{H^1(M,\mathbb{R})}
			+ |\delta_{\varepsilon,E_{p_1,m}} - \delta_{\varepsilon,E_{p_2,m}}|
		\Big].
	\end{multline*}
	The interested reader can find more details on this last inequality
	in~\cite[Proposition 4.14]{benci-nardulli-osorio-piccione}.
	Let us observe that, as long as $m < m_0$,
	the hypersurface $\Sigma_{p,m} = \partial E_{p,m} \cap M$
	is smooth and diffeomorphic to a hemisphere.
	The proof ends by observing that,
	as $p_1 \to p_2$, we have
	\[
		\norm{d_{M\setminus E_{p_1,m}}
		- d_{M\setminus E_{p_2,m}}}_{H^1(M,\mathbb{R})} \to 0
		\quad\text{and}\quad
		|\delta_{\varepsilon,E_{p_1,m}} - \delta_{\varepsilon,E_{p_2,m}}| \to 0,
	\]
	the last convergence obtained as a consequence of 
	the Implicit Function Theorem applied 
	to~\eqref{eq:def-delta-epsilon-Omega},
	where $E_{p,m}$ substitutes $\Omega$,
	to define the $C^1$ map 
	$p \mapsto \delta_{\varepsilon,E_{p,m}}$.
\end{proof}

\subsection{Barycenter map}

The map
$\mathcal{B}\colon \mathcal{E}_{\varepsilon,m}^c \to \partial M$
is obtained by combining the $\Gamma$-convergence result
with the information of the isoperimetric problem given by
Proposition~\ref{prop:almost-isoperimetric-in-a-small-ball},
which ensures that the ``almost'' minimizers of $\mathcal{E}_0$
have almost all their mass inside a small ball centered on a point
of the boundary $\partial M$.
While the map $\mathcal{P}_{\varepsilon,m}$ has been constructed
by using the {lim-sup} part of the $\Gamma$-convergence result,
namely Theorem~\ref{theorem:G-convergence},
to define $\mathcal{B}$ we rely on the
{lim-inf} statement 
and on the compactness result ensured by
Theorem~\ref{theorem:recovery-sequence}.

Let us begin by introducing some new notation.
Let $\mathrm{dist}_g\colon M\times M \to \mathbb{R}$ be the distance
function induced by $g$, hence
\begin{equation}
	\label{eq:def-dist_g}
	\mathrm{dist}_g(x,y) \coloneqq
	\inf\left\{
		\left(\int_{0}^1 g(\dot{\gamma},\dot{\gamma})\de s\right)^{1/2}:
		\gamma \in H^1([0,1],M), 
		\gamma(0) = x, \gamma(1) = y
	\right\}.
\end{equation}
With this notation, we set 
$\mathrm{diam}_g(M) = \max\{ \mathrm{dist}_g(x,y): x,y \in M\}$,
and for any $p\in M$ and $r \in (0, \mathrm{diam}_g(M))$, let 
\begin{equation}
	\label{eq:def-Bxr}
	B(p,r) \coloneqq
	\left\{
		x \in M: 
		\mathrm{dist}_g(p,x) \le r
	\right\}.
\end{equation}
Without loss of generality, 
from now on we assume 
that $M$ is isometrically embedded in $\R^{\tilde{n}}$,
for some $\tilde{n} \geq n$.
This will allow us to construct the map $\mathcal{B}$
as the composition of the ``extrinsic'' barycenter map
and the nearest point projection on $\partial M$.
More formally, let $(\partial M)_r \subset \mathbb{R}^{\tilde{n}}$
be the following set
\begin{equation}
	\label{eq:def-partialM_r}
	(\partial M)_r
	\coloneqq
	\left\{
		z \in \mathbb{R}^{\tilde{n}}:
		\dist_{\mathbb{R}^{\tilde{n}}}(z,\partial M) \le r
	\right\},
\end{equation}
and let us denote by 
$\pi_{\partial M}\colon (\partial M)_r \to \partial M$
the nearest point projection onto $\partial M$.
Notice that $\pi_{\partial M}$
is well defined when $r$ is sufficiently small,
thanks to the compactness assumption on $\partial M$.
We will combine $\pi_{\partial M}$ with the following map
\begin{equation}
	\label{eq:def-beta*}
	\beta^*\colon
	u \in L^1(M,\R) \to
	\frac{ \int_M x\lvert u(x) \rvert \de v_g}{\int_M \lvert u(x) \rvert \de v_g}
	\in \R^{\tilde{n}},
\end{equation}
which gives the center of mass of a function in $\mathbb{R}^{\tilde{n}}$.
\begin{lemma}
	\label{lem:beta*-continuity}
	The map $\beta^*\colon L^1(M,\mathbb{R}) \to \mathbb{R}^{\tilde{n}}$
	is continuous.
\end{lemma}
\begin{proof}
	For every $u_1,u_2 \in L^1(M,R)$,
	setting $v_1 = \int_{M}|u_1|\de v_g$
	and $v_2 = \int_{M}|u_2|\de v_g$,
	we have
	\begin{multline*}
		\norm*{
			\beta^*(u_1) - \beta^*(u_2)
		}_{\mathbb{R}^{\tilde{n}}}
		=
		\norm*{
			\frac{ \int_M x|u_1(x)| \de v_g}{v_1}
			-
			\frac{ \int_M x|u_2(x)| \de v_g}{v_2}
		}_{\mathbb{R}^{\tilde{n}}}\\
		\le 
		\frac{2\norm{x}_{\infty}}{v_1}
		\int_M \left| |u_1| - \frac{v_1}{v_2}|u_2|\right|\de v_g,
	\end{multline*}
	where
	$\norm{x}_{\infty} = \max\{\norm{x}_{\mathbb{R}^{\tilde{n}}}: x \in M\}$.
	By using Lebesgue's dominated convergence
	and H\"older's inequality, we get
	\[
		\frac{2\norm{x}_{\infty}}{v_1}
		\int_M \left| |u_1| - \frac{v_1}{v_2}|u_2|\right|\de v_g
		\to 0
		\quad\text{as}\quad \norm{u_1 - u_2}_{L^1(M,\mathbb{R})} \to 0,
	\]
	so we obtain that 
	\[
		\norm*{
			\beta^*(u_1) - \beta^*(u_2)
		}_{\mathbb{R}^{\tilde{n}}} \to 0,
		\quad\text{as}\quad \norm{u_1 - u_2}_{L^1(M,\mathbb{R})} \to 0,
	\]
	which concludes the proof.
\end{proof}

In order to define
$\mathcal{B}\colon \mathcal{E}_{\varepsilon,m}^{\sigma I_M(m) + \theta m}
\to \partial M$,
one needs to show that
for every $r > 0$ there exist
$m$ and $\varepsilon$ sufficiently small
such that
\begin{equation}
	\label{property_barycenter}
	\beta^*(\mathcal{E}_{\varepsilon,m}^{\sigma I_M(m)+\theta m})
	\subset (\partial M)_r.
\end{equation}
In other words, we need to prove that if $m$ and $\varepsilon$
are sufficiently small, then the barycenter of every function in the 
sublevel $\mathcal{E}_{\varepsilon,m}^{\sigma I_M(m) + \theta m}$
lies in a small tubular neighbourhood of $\partial M$.
Then, choosing $r$ sufficiently small as well,
one can define
$ \mathcal{B}\colon \mathcal{E}^{\sigma I_M(m)+\theta m}_{\eps,m} \to \partial M$
as follows:
\begin{equation}
	\label{eq:def-Barycenter}
	\mathcal{B} = \pi_{\partial M}\circ \beta^*.
\end{equation}
To this end, we need to ensure
that if the volume $m$ is sufficiently small
then any function
in the sublevel $\mathcal{E}_{\varepsilon,m}^{\sigma I_M(m) + \theta m}$
has almost all its mass in a small ball centered 
on a point of the boundary $\partial M$
(see Proposition~\ref{prop:almost-isoperimetric-in-a-small-ball}).
The proof of this result is based
on~\cite[Theorem 1.2]{antonelli2022},
which gives a compactness result
for sequences of sets with uniformly bounded volume and perimeter.
For the sake of presentation, we restate here that theorem.
If $(X,\de,\mathcal{H}^{n})$ is a metric measure space
(where $\mathcal{H}^n$ denots the $n$--dimensional Hausdorff measure)
and $\Omega\subset X$ is a Borel set,
we remind that the definition of perimeter 
of $\Omega$ used in the 
following theorem is 
\[
	P(\Omega,X)\coloneqq
	\inf
	\left\{
		\liminf_{i}
		\int_X
		\mathrm{lip} f_i\de\mathcal{H}^{n}
		:
		f_i \in \mathrm{Lip}_{\mathrm{loc}}(X),
		f_i \to \mathbf{1}_{\Omega}
		\text{ in }L^1_{\text{loc}}(X,\mathcal{H}^n)
	\right\},
\]
where $\mathrm{lip}f(x) \coloneqq \limsup_{y \to x}\frac{|f(y)-f(x)|}{\de(x,y)}$.
Moreover,
we recall that a metric space 
$(X,\de,\mathcal{H}^n)$
is an $\mathrm{RCD}(\kappa,n)$ space
if, loosely speaking,
it has synthetic notions of Ricci curvature bounded below by $\kappa \in \mathbb{R}$
and dimension bounded above by $n \in (0, +\infty]$
(see~\cite{antonelli2022} for more details).

\begin{theorem}[Theorem 1.2 of~\cite{antonelli2022}]
	\label{theorem:antonelli2022}
	Let $\kappa\in \mathbb{R}$ and $n \ge 2$.
	Let $(X_k,\de_k,\mathcal{H}^{n}_k)$
	be a sequence of $\mathrm{RCD}(\kappa,n)$ spaces
	and let $\Omega_k \subset X_k$
	be bounded sets of finite perimeter such that
	$\sup_{k} \big(P(\Omega_k,X_k) + \mathcal{H}^{n}_k(\Omega_k)\big)<+\infty$.
	Assume there exists $v_0 > 0$ such that 
	$\mathcal{H}^n_k(B(x,1))\ge v_0$ for every $x \in X_k$
	and for every $k\in\mathbb{N}$.
	Then, up to subsequences, there exists a nondecreasing
	sequence $(N_k)_{k \in \mathbb{N}}\subset \mathbb{N}$,
	$N_k \ge 1$, points
	$p_{k,i} \in X_k$, with $1 \le i \le N_k$
	and pairwise disjoint subsets
	$\Omega_{k,i}\subset \Omega_{k}$
	such that:
	\begin{enumerate}[i)]
		\item $\lim_{k \to \infty} \de_k(p_{k,i},p_{k,j}) = +\infty$
			for any $i \ne j \le \bar{N}$,
			where $\overline{N} \coloneqq \lim_{k \to \infty}
			N_k \in \mathbb{N} \cup \{+\infty\}$;
		\item For every $1 \le i < \overline{N} + 1$,
			the sequence 
			$(X_k,\de_k,\mathcal{H}^{n},p_{k,i})$
			converges in the pointed measured
			Gromov-Hausdorff topology
			to a pointed $\mathrm{RCD}(\kappa,n)$ space
			$(Y_{i},\dist_{Y_i},\mathcal{H}^n,p_i)$
			as $k \to \infty$;
		\item There exists $F_i \subset Y_i$ such that 
			$\Omega_{k,i}\to F_i$ as $k \to \infty$ in $L^1$--strong topology
			and the following relations hold:
			\begin{equation*}
				\lim_{k\to \infty}
				\mathcal{H}^n_k(\Omega_k) = 
				\sum_{i = 1}^{\overline{N}}\mathcal{H}^{n}(F_i),
			\end{equation*}
			and
			\begin{equation*}
				\sum_{i = 1}^{\overline{N}}
				P(F_i,Y_i)
				\le  \liminf_{k \to \infty} P(\Omega_k,X_k).
			\end{equation*}
	\end{enumerate}
	Moreover, if $\Omega_k$ is an isoperimetric set in $X_k$
	for any $k$, 
	then $F_i$ is an isoperimetric set in $Y_i$
	for any $i<\overline{N} + 1$ and
	\begin{equation}
		\label{eq:lastResult-ANP}
		P(F_i,Y_i) = \lim_{k\to \infty}P(\Omega_{k,i},X_k)
	\end{equation}
	for any $i<\overline{N}+1$.
\end{theorem}

\begin{remark}
	\label{rem:how-use-Antonelli}
	In the following, we use the previous theorem 
	on the more simple case of Riemannian manifolds.
	We want to highlight that when 
	the $\mathrm{RCD}$ space is $(\mathrm{int}(M),g,\mathcal{H}^n)$,
	the perimeter of a measurable set $\Omega\subset \mathrm{int}(M)$
	coincides with the 
	$(n-1)$--dimensional Hausdorff measure of its relative boundary,
	hence 
	\[
		P(\Omega,\mathrm{int}(M)) = \mathcal{H}^{n-1}(\partial^*\Omega\cap\mathrm{int}(M)).
	\]
	When we consider the space
	$(M,g,\mathcal{H}^n)$, 
	the whole perimeter of the set $\Omega\subset M$
	is considered, 
	namely
	\[
		P(\Omega,M) = \mathcal{H}^{n-1}(\partial^*\Omega).
	\]
\end{remark}

\begin{proposition}
	\label{prop:almost-isoperimetric-in-a-small-ball}
	There exists $\mu = \mu(M,g)> 0$ such that the following property holds.
	For every almost isoperimetric sequence
	$(u_k)_{k\in \mathbb{N}}\subset  L^1(M,\mathbb{R})$
	with volumes $m_k = \int_{M} u_{k}\de v_g \to0$,
	i.e., 
	\begin{equation}
		\label{eq:almost-isoperimetric-in-a-small-ball-HP}
		\lim_{k \to \infty}
		\frac{\mathcal{E}_{0}(u_k)}{\sigma I_M(m_k)} = 1,
	\end{equation}
	there exists a sequence $(p_k)_{k \in \mathbb{N}} \subset \partial M$
	such that
	\begin{equation}
		\label{eq:almost-isoperimetric-in-a-small-ball}
		\lim_{k \to +\infty}
		\frac{1}{m_k}
		\left(
			\int_{M\setminus B(p_k,\mu m_{k}^{1/n})} u_{k} \de v_g
		\right)
		= 0.
	\end{equation}
\end{proposition}

\begin{proof}[Proof of Proposition~\ref{prop:almost-isoperimetric-in-a-small-ball}]

	For every fixed $k \in \mathbb{N}$,
	let $\Omega_{k} \in \mathcal{C}_g(M)$ be such that
	$u_k = \textbf{1}_{\Omega_{k}}$.
	Note that, given that $u_k \in L^1(M, \mathbb{R})$,
	we can select any $\Omega_k$ in such a way that $\Omega_k$
	is a subset of the interior of $M$.
	Let us define the following sequence of manifolds
	of bounded geometry:
	\[
		(X_k,g_k)
		\coloneqq 
		(\mathrm{int}(M),m_{k}^{-1/n}g).
	\]
	In other words, we rescale the metric on $M$ by a factor 
	that depends on $m_k$ in such a way that 
	\[
		\mathcal{H}_{k}^{n}(\Omega_k) = 
		\int_{X_k} u_k \de v_{g_k} = 1,
		\qquad \forall k \in \mathbb{N},
	\]
	where $\mathcal{H}_{k}^{n}$ denotes the 
	$n$--dimensional Hausdorff measure on $(X_k,g_k)$.
	Moreover, by~\eqref{eq:IM-to-IR+}
	and~\eqref{eq:almost-isoperimetric-in-a-small-ball-HP}
	and recalling that
	$\mathcal{E}_0(u_{k})
	= \sigma \mathcal{H}^{n-1}(\partial^*\Omega_{k}\cap \mathrm{int}(M))$,
	we obtain the existence of a constant $C$ such that
	\[
		P(\Omega_k,X_k)
		= \mathcal{H}_{k}^{n-1}(\partial^*\Omega_k\cap\mathrm{int}(M))
		<  I_{\mathbb{R}^n_{+}}(1) + C,
		\qquad \forall k \in \mathbb{N}.
	\]
	Moreover, we have also the following equality:
	\begin{equation}
		\label{eq:almost-isoperimetric-in-a-small-ball-proof1}
		\lim_{k \to \infty}
		\mathcal{H}_{k}^{n-1}(\partial^*\Omega_k\cap\mathrm{int}(M))
		= I_{\mathbb{R}^n_+}(1).
	\end{equation}
	This means that the sequence $\{(X_k,g_k,\mathcal{H}^n_k)\}_{k \in \mathbb{N}}$
	and $\Omega_k \subset M_k$ satisfy the hypotheses of
	Theorem~\ref{theorem:antonelli2022}
	and so,
	up to subsequences, for every $k$ there exist
	$N_k \in \mathbb{N}$, with $N_k \ge 1$,
	and points $p_{k,i} \in M_k$ such that 
	for every $i \in 1,\dots,N_k$
	the space $(X_{k},\dist_{g_k},\mathcal{H}^{n}_{k},p_{k,i})$
	converges in the Gromov-Hausdorff topology
	to a pointed $\mathrm{RCD}(\kappa,n)$
	as $k\to \infty$,
	that we denote by
	$(M^i_{\infty},\de^i_{\infty},\mathcal{H}^{n,i}_{\infty},p^i_{\infty})$.
	This space could be either 
	the Euclidean semi-space
	$(\mathbb{R}^{n}_+,\de,\mathcal{H}^{n},0)$
	or the Euclidean space 
	$(\mathbb{R}^{n},\de,\mathcal{H}^{n},0)$,
	depending on if $(p_{k,i})_{k}$ converges
	to a point in $\partial M$ or not, respectively.
	Moreover, setting $\overline{N} = \lim_{k\to \infty} N_k$,
	for every $i = 1,\dots,\overline{N}$ 
	there exists $\Omega_{\infty}^i\subset M^{i}_{\infty}$
	such that 
	\begin{equation}
		\label{eq:almost-isoperimetric-in-a-small-ball-proof2}
		\sum_{i = 1}^{\overline{N}}\mathcal{H}^{n,i}_{\infty}(\Omega_{\infty}^i)
		= \lim_{k \to \infty}\mathcal{H}^{n}_k(\Omega_k)
		= 1
	\end{equation}
	and, using also~\eqref{eq:almost-isoperimetric-in-a-small-ball-proof1},
	we have
	\begin{equation}
		\label{eq:almost-isoperimetric-in-a-small-ball-proof3}
		\sum_{i = 1}^{\overline{N}}
		\mathcal{H}^{n-1,i}_{\infty}(\partial^*\Omega_{\infty}^i \cap \mathrm{int}(M^{i}_{\infty})
		) \le
		\liminf_{k \to \infty} 
		\mathcal{H}_{k}^{n-1}(\partial^*\Omega_{k}\cap\mathrm{int}(M))
		= 
		I_{\mathbb{R}^n_+}(1).
	\end{equation}
	By using contradiction arguments,
	this last inequality implies both that
	$\overline{N} = 1$ and that 
	$(M^1_{\infty},\de^1_{\infty},\mathcal{H}^n_{\infty},p^1_{\infty})$
	is actually the Euclidean semispace
	$(\mathbb{R}^{n}_+,\de,\mathcal{H}^{n},0)$,
	since $I_{\mathbb{R}^n_+}(1) < I_{\mathbb{R}^n}(1)$.
	As a consequence,
	by using~\eqref{eq:almost-isoperimetric-in-a-small-ball-proof2},
	we get the existence of $\mu > 0$
	and a sequence $p_k \in \partial M$
	such that~\eqref{eq:almost-isoperimetric-in-a-small-ball}
	holds.
\end{proof}

Using the same constants $\theta$ and $\mu$ 
given by Proposition~\ref{prop:photography-sublevel} 
and Proposition~\ref{prop:almost-isoperimetric-in-a-small-ball},
respectively,
we obtain the following result,
which ensures that all the functions 
in the sublevel that contains the image of the photography 
map have their mass concentrated near the boundary. 

\begin{lemma}
	\label{lem:in-a-small-semisphere}
	Assume that \ref{asu_nondegenerate} and \ref{asu_coercive} hold. For every $\alpha \in (0,1)$,
	there exists $m_{\alpha} = m_{\alpha}(M,g,W,\theta,\alpha) > 0$
	such that for every $m \in (0,m_{\alpha})$
	there exists $\varepsilon_{\alpha} = \varepsilon_{\alpha}(M,g,W,\theta,\alpha,m) > 0$
	such that for every $\varepsilon \in (0,\varepsilon_{\alpha})$
	and for any 
	$u \in \mathcal{E}_{\varepsilon,m}^{\sigma I_M(m) + \theta m}$
	there exists a point $p_u \in \partial M$
	such that
	\begin{equation}
		\label{eq:mass-in-a-small-semisphere}
		\int_{M\setminus B(p_u,\mu m^{1/n})} |u| \de v_g
		\le \alpha m.
	\end{equation}
\end{lemma}
\begin{proof}
	Arguing by contradiction, there exists a sequence 
	$(m_{i})_{i \in \mathbb{N}}$ such that 
	$m_i \to 0^+$ and for every $i \in \mathbb{N}$
	there exist two sequences
	$(\varepsilon_{i,j})_{j\in \mathbb{N}} \in \mathbb{R}$
	and
	$(u_{i,j})_{j\in \mathbb{N}} \in \mathcal{E}_{\varepsilon_{i,j},m_i}^{c_i}$,
	with $c_i = I_M(m_i) + \theta m_i$,
	such that for every $i \in \mathbb{N}$
	we have
	$\varepsilon_{i,j}\to 0^+$ as $j \to \infty$
	and
	\begin{equation}
		\label{eq:in-a-small-semisphere-proof1}
		\int_{M\setminus B_g(p,\mu m_i^{1/n})} |u_{i,j}| \de v_g
		> \alpha m_i,
		\qquad \forall p \in \partial M,\, \forall j \in \mathbb{N}.
	\end{equation}
	Since for any fixed $i \in \mathbb{N}$ we have
	$\mathcal{E}_{\varepsilon_{i,j},m_i}(u_{i,j}) \le c_i$
        for every $j \in \mathbb{N}$,
	we can apply Theorem~\ref{theorem:recovery-sequence}
	to obtain
	a sequence of characteristic functions
	$(u_{0,i})_{i} \subset L^1(M,\mathbb{R})$
	such that, for every $i$,
	we have
	$ \lim_{j \to \infty}
	\lVert u_{i,j} - u_{0,i}\rVert_{L^1(M,\mathbb{R})} = 0$,
        up to subsequences
        (see also Remark~\ref{rem:convergence-valueE0}).
	Hence, for every $i$ there exists $j_i$ 
	such that 
	\begin{equation}
		\label{eq:in-a-small-semisphere-proof2}
		\int_{M} \big(|u_{i,j_{i}}| - u_{0,i}\big)\de v_g
		\le \frac{\alpha}{4} m_i,
	\end{equation}
	and
	we have also
	$\int_{M} u_{0,i} \de v_g = m_i$
	and $\mathcal{E}_0(u_{0,i}) \le c_i$,
	namely
	\[
		I_M(m_i) \le \frac{\mathcal{E}_0(u_{0,i})}{\sigma} \le I_M(m_i) + \theta m_i.
	\]
	Using also~\eqref{eq:omega-m-oIm},
	we obtain
	\[
		\lim_{i \to \infty}
		\frac{\mathcal{E}_0(u_{0,i})}{\sigma I_M(m_i)} = 1,
	\]
	so
	we can apply Proposition~\ref{prop:almost-isoperimetric-in-a-small-ball}
	to obtain the existence of a sequence
	$(p_i)_{i \in \mathbb{N}} \subset \partial M$
	such that
	\[
		\lim_{i \to \infty}
		\frac{1}{m_i}
		\left(
			\int_{M\setminus B(p_i,\mu m_{i}^{1/n})} u_{0,i}\, \de v_g
		\right)
		= 0,
	\]
	so there exists $i_0$ such that 
	\begin{equation}
		\label{eq:in-a-small-semisphere-proof3}
		\int_{M\setminus B(p_i,\mu m_{i}^{1/n})} u_{0,i}\, \de v_g
		\le \frac{\alpha}{4}m_i,
		\qquad \forall i \ge i_0.
	\end{equation}
	As a consequence, 
	combining~\eqref{eq:in-a-small-semisphere-proof2}
	and~\eqref{eq:in-a-small-semisphere-proof3},
	for every $i > i_0$
	we obtain 
	\begin{multline*}
		\label{Eq:almostAllVolumInBall-proof7}
		\int_{M\setminus B(p_i,\mu m_i^{1/n})}
		|u_{i,j_i}|\de v_g
		=
		\int_{M\setminus B(p_i,\mu m_i^{1/n})}
		\Big(|u_{i,j_i}| - u_{0,i}\Big) \de v_g \\
		+
		\int_{M\setminus B(p_i,\mu m_i^{1/n})}u_{0,i}\, \de v_g
		\\
		\le 
		\int_{M}
		\big(|u_{i,j_i}| - u_{0,i}\big) \de v_g
		+
		\int_{M\setminus B(p_i,\mu m_i^{1/n})}u_{0,i}\, \de v_g
		\le \frac{\alpha}{2}m_i,
	\end{multline*}
	which contradicts~\eqref{eq:in-a-small-semisphere-proof1}.
\end{proof}

Using the previous result,
we can ensure that if $m$ and $\varepsilon$
are sufficiently small,
the extrinsic barycenter of any function 
in the sublevel of $\mathcal{E}_{\varepsilon,m}$
that contains the image of the photograpy map
is near the boundary of the manifold.
More formally,
setting
\[
	\mathrm{diam}_{\mathbb{R}^{\tilde{n}}}(M)
	\coloneqq
	\max\big\{\lVert x - y \rVert_{\mathbb{R}^{\tilde{n}}}:
	x,y \in M\big\},
\]
we have the following result.
\begin{lemma}
	\label{lem:beta*-near-partialM}
	Assume that \ref{asu_nondegenerate} and \ref{asu_coercive} hold. For every $r > 0$, there exists
	$m_2 = m_2(M,g,r,\mathrm{diam}_{\mathbb{R}^{\tilde{n}}}(M)) > 0$
	such that for every $m \in (0,m_2)$
	there exists
	$\varepsilon_2 = \varepsilon_2(M,g,r,m) > 0$
	such that for every $\varepsilon \in (0,\varepsilon_2)$
	and any $u \in \mathcal{E}_{\varepsilon,m}^{\sigma I_M(m) + \theta m}$
	we have
	$\beta^*(u)\in (\partial M)_r$.
\end{lemma}
\begin{proof}
	Recalling the definition of $\beta^*$ given by~\eqref{eq:def-beta*},
	for every $u \in H^1(M,\mathbb{R})$ and $y \in M$,
	let us define 
	\[
		\rho(u,y) =  
		\frac{|u(y)|}{\int_{M}| u(x)| \de v_g},
	\]
	so that 
	\[
		\beta^*(u) = \int_{M} x \rho(u,x)\de v_g
		\quad\text{and}\quad
		\int_M \rho(u,x) \de v_g = 1, \, \forall u \in H^1(M,\mathbb{R}).
	\]

	Let us choose $\alpha > 0$ such that
	\[
		\alpha \le \frac{r}{2\mathrm{diam}_{\mathbb{R}^{\tilde{n}}}(M) },
	\]
	and let $m_{\alpha}$ be given by
	Lemma~\ref{lem:in-a-small-semisphere}.
	Let us choose $m \in (0,m_{\alpha})$
	and let $\varepsilon \in (0,\varepsilon_{\alpha})$.
	By Lemma~\ref{lem:in-a-small-semisphere},
	for every $u \in \mathcal{E}_{\varepsilon,m}^{\sigma I_M(m) + \theta m}$
	there exists $p_u \in \partial M$
	such that~\eqref{eq:mass-in-a-small-semisphere} holds.
	As a consequence, for every $u \in \mathcal{E}_{\varepsilon,m}^{\sigma I_M(m) + \theta m}$
	we obtain
	\begin{multline*}
		\norm{\beta^*(u) - p_u}_{\mathbb{R}^{\tilde{n}}}
		= \norm*{\int_{M}\big(x - p_u\big)\rho(u,x)\de v_g}_{\mathbb{R}^{\tilde{n}}}\\
		\le \norm*{
			\int_{B(p_u,\mu m^{1/n})}
			\big(x - p_u\big)\rho(u,x)\de v_g
		}_{\mathbb{R}^{\tilde{n}}}
		+ \norm*{
			\int_{M\setminus B(p_u,\mu m^{1/n})}
			\big(x - p_u\big)\rho(u,x)\de v_g
		}_{\mathbb{R}^{\tilde{n}}}\\
		\le \mu m^{1/n} + \alpha \mathrm{diam}_{\mathbb{R}^{\tilde{n}}}(M)
		\le \mu m^{1/n} + \frac{r}{2}.
	\end{multline*}
	As a consequence, choosing $m_2 \in (0,m_{\alpha})$
	such that 
	\[
		\mu m_2^{1/n} \le \frac{r}{2},
	\]
	for all $m \in (0, m_2)$
	and for all $\varepsilon \in (0,\varepsilon_2)$,
	with $\varepsilon_{2} = \varepsilon_{\alpha}$,
	we obtain 
	\[
		\norm{\beta^*(u) - p_u}_{\mathbb{R}^{\tilde{n}}}
		\le r,
		\qquad \forall u \in \mathcal{E}_{\varepsilon,m}^{I_M(m) + \theta m},
	\]
	and we are done.
\end{proof}

\begin{remark}
	\label{rem:def-r1}
	By Lemma~\ref{lem:beta*-near-partialM}
	and the compactness of $\partial M$,
	there exists $r_1 > 0$ such that,
	choosing $m_2(M,g,r_1,\mathrm{diam}_{\mathbb{R}^{\tilde{n}}}) > 0$
	and $\varepsilon_2(M,g,r_1,m) > 0$ as in the lemma,
	for every $m \in (0,m_2)$
	and $\varepsilon \in (0,\varepsilon_2)$,
	the map $\mathcal{B} = \pi_{\partial M}\circ \beta^*\colon 
	\mathcal{E}_{\varepsilon,m}^{\sigma I_M(m) + \theta m}\to \partial M$
	is well defined and continuous.
\end{remark}

\subsection{Conclusion of the proof}
\label{sect_Neumann_completed}

The following two results, namely,
Lemma~\ref{lem:barycenter+photography}
and Lemma~\ref{lem:homotopy-barycenter+photography},
demonstrate that the barycenter of a function obtained
by applying the photography map to a point $p \in \partial M$
is in close proximity to $p$ itself
when $m$ and $\varepsilon$ are sufficiently small.
As a result,
taking into account the definition of $\mathcal{B}$
provided in~\eqref{eq:def-Barycenter},
the composition
$\mathcal{B}\circ \mathcal{P}_{\varepsilon,m}\colon \partial M \to \partial M$
closely approximates the identity map.

\begin{lemma}
	\label{lem:barycenter+photography}
	Assume that \ref{asu_nondegenerate} and \ref{asu_coercive} hold. For every $r \in ]0,r_1[$
	there exists $m_3 = m_3(M,g,r) > 0$
	such that for every $m \in (0,m_3)$
	there exists $\varepsilon_3 = \varepsilon_3(M,g,r,m) > 0$
	such that
	\begin{equation}
		\label{eq:bound-on-dist-beta-photo}
		\norm{\beta^*(\mathcal{P}_{\varepsilon,m}(p)) - p}_{\mathbb{R}^{\tilde{n}}}
		\le r,
		\qquad \forall p \in \partial M.
	\end{equation}
\end{lemma}
\begin{proof}
	We recall that for every $p \in \partial M$
	and every $m \in (0,\vol(M))$,
	the radius $r_{p,m}$ is the one such that
	\[
		\int_{E_{p,r_{p,m}}} 1 \de v_g = m.
	\]
	By the compactness of $\partial M$,
	there exists $m_3 = m_3(M,g,r) > 0$
	such that for every
	$m \in (0,m_3)$
	and for every $p \in \partial M$
	we have $r_{p,m} < r/2$.
	As a consequence,
	for every $x \in E_{p,r_{p,m}} \subset M$ we have
	\begin{equation}
		\label{eq:bound-on-dist-beta-photo-proof1}
		\norm{x - p}_{\mathbb{R}^{\tilde{n}}} < \frac{r}{2}.
	\end{equation}
	Moreover, we recall that for every $m,\varepsilon > 0$
	the function $\mathcal{P}_{\varepsilon,m}(p)$
	is defined as the $\varepsilon$--approximation
	of $u_{0,x}^m	= \mathbf{1}_{E_{p,m}}$
	and that 
	$u_{\varepsilon,p}^m$
	converges to $u_{0,x}^m$
	in the $L^1(M,\mathbb{R})$ norm as $\varepsilon \to 0$.
	Hence, using also the continuity of $\beta^*$
	with respect to the $L^1$--norm ensured by 
	Lemma~\ref{lem:beta*-continuity}
	and again the compactness of $\partial M$,
	we obtain the existence of a constant 
	$\varepsilon_3 = \varepsilon_3(M,g,r,m) > 0$
	such that
	for every $\varepsilon \in (0,\varepsilon_3)$
	and for every $p \in \partial M$
	the following holds:
	\begin{equation}
		\label{eq:bound-on-dist-beta-photo-proof2}
		\norm{
			\beta^*(\mathcal{P}_{\varepsilon,m}(p))
			- \beta^*(u_{0,p}^m)
		}_{\mathbb{R}^{\tilde{n}}}
		=
		\norm{
			\beta^*(u_{\varepsilon,p}^m)
			- \beta^*(u_{0,p}^m)
		}_{\mathbb{R}^{\tilde{n}}}
		\le \frac{r}{2}.
	\end{equation}
	As a consequence, using
	both~\eqref{eq:bound-on-dist-beta-photo-proof1}
	and\eqref{eq:bound-on-dist-beta-photo-proof2},
	for every $m \in (0,m_3)$
	and for every $\varepsilon \in (0,\varepsilon_3)$
	we obtain
	\begin{multline*}
		\norm{\beta^*(\mathcal{P}_{\varepsilon,m}(p)) - p}_{\mathbb{R}^{\tilde{n}}}
		\le
		\norm{
			\beta^*(\mathcal{P}_{\varepsilon,m}(p))
			- \beta^*(u_{0,p}^m)
		}_{\mathbb{R}^{\tilde{n}}}
		+
		\norm{
			\beta^*(u_{0,p}^m) - p
		}_{\mathbb{R}^{\tilde{n}}}\\
		\le \frac{r}{2} 
		+
		\frac{1}{\int_M |u_{0,p}^m| \de v_g}
		\int_M\Big(
			\norm*{ x - p
			}_{\mathbb{R}^{\tilde{n}}}
			|u_{0,p}^m|
		\Big)\de v_g
		\le r,
		\qquad \forall p \in \partial M,
	\end{multline*}
	which concludes the proof.
\end{proof}

\begin{lemma}
	\label{lem:homotopy-barycenter+photography}
	Assume that \ref{asu_nondegenerate} and \ref{asu_coercive} hold.
        There exists $m^* = m^*(M,g,W) > 0$ such that 
	for any $m \in (0, m^*)$ there exists
	$\varepsilon^* = \varepsilon^*(M,g,W,m) > 0$ such that
	for any $\varepsilon \in (0, \varepsilon^*)$
	the composition map 
	\[
		\mathcal{B} \circ \mathcal{P}_{\varepsilon,m}
		\colon \partial M \to \partial M
	\]
	is well defined and homotopic to the identity map.
\end{lemma}
\begin{proof}
	For $p \in \partial M$,
	let $\mathrm{exp}^{\partial M}_{p}$ be the exponential map
	of $\partial M$ at $p$,
	and let $B^{\partial M}_{g}(p,{R})$ be the geodesic ball of center $p$
	and radius $R$ on $\partial M$.
	We denote by
	$\mathrm{inj}(\partial M)$
	the \emph{injectivity radius} of $\partial M$, that is
	\begin{equation*}
		\mathrm{inj}(\partial M)\coloneqq
		\inf_{p \in \partial M}\sup\big\{R>0
			\text{ s.t. }
			\mathrm{exp}^{\partial M}_{p}\colon B^{\partial M}_{g}(p,{R})
		\to \partial M \mbox{ is a diffeomorphism}\big\}
	\end{equation*}
	which is positive since $\partial M$ is a compact manifold.
	Moreover, 
	the compactness of $\partial M$ implies
	that there exists a constant $C_{\partial M}>0$
	such that 
	\begin{equation*}
		\dist_g(p,q) \leq C_{\partial M}\norm{p - q}_{\mathbb{R}^{\tilde{n}}}
		\qquad \forall p,q \in \partial M.
	\end{equation*}
	Recalling the definition of $r_1 > 0$ given in Remark~\ref{rem:def-r1},
	let
	\[
		r^* =  \min\big\{r_1,\mathrm{inj}(\partial M)/(4C_{\partial M})\big\},
	\]
	let $m_2 = m_2(M,g,r^*,\mathrm{diam}_{\mathbb{R}^{\tilde{n}}}(M)) > 0$
	be defined by Lemma~\ref{lem:beta*-near-partialM},
	and for every $m \in (0,m_2)$,
	set $\varepsilon_2 = \varepsilon_2(M,g,r^*,m) > 0$
	in the same way.
	Moreover, let $m_3 = m_3(M,g,r^*) > 0$
	and $\varepsilon_3 = \varepsilon_3(M,g,r^*,m)$
	be similarly defined by Lemma~\ref{lem:barycenter+photography}.
	Recalling Proposition~\ref{prop:photography-sublevel},
	let $m^* = \min\{m_1,m_2,m_3\}$
	and for every $m \in (0,m^*)$
	set $\varepsilon^* = \min\{\varepsilon_1,\varepsilon_2,\varepsilon_3\}$.
	By Proposition~\ref{prop:photography-sublevel},
	for any $p \in \partial M$
	we have 
	$\mathcal{P}_{\varepsilon,m}(p)
	\in \mathcal{E}_{\varepsilon,m}^{\sigma I_M(m) + \theta m}$,
	so we obtain
	\begin{multline}
		\label{eq:homotopy-barycenter+photography-proof1}
		\mathrm{dist}_g\big(\mathcal{B}(\mathcal{P}_{\varepsilon,m}(p)),p\big)
		\le C_{\partial M}\norm{\pi_{\partial M}
			(\beta^*
		(\mathcal{P}_{\varepsilon,m}(p)))-p}_{\mathbb{R}^{\tilde{n}}}\\
		\le C_{\partial M}\big(
			\norm{\pi_{\partial M}(\beta^*(\mathcal{P}_{\varepsilon,m}(p))) -
			\beta^*(\mathcal{P}_{\varepsilon,m}(p))}_{\mathbb{R}^{\tilde{n}}}
			+ 
			\norm{\beta^*(\mathcal{P}_{\varepsilon,m}(p))
			- p}_{\mathbb{R}^{\tilde{n}}}
		\big).
	\end{multline}
	Since $m^*$ and $\varepsilon^*$
	are, respectively,
	less than $m_2$ and $\varepsilon_2$ 
	defined as in Lemma~\ref{lem:beta*-near-partialM},
	we have
	\[
		\norm{\pi_{\partial M}(\beta^*(\mathcal{P}_{\varepsilon,m}(p))) -
		\beta^*(\mathcal{P}_{\varepsilon,m}(p))}_{\mathbb{R}^{\tilde{n}}}
		\le r^*,
	\]
	and, by applying in an analogous manner
	Lemma~\ref{lem:barycenter+photography},
	we obtain also
	\[
		\norm{\beta^*(\mathcal{P}_{\varepsilon,m}(p)) - p}_{\mathbb{R}^{\tilde{n}}}
		\le r^*.
	\]
	Hence, from~\eqref{eq:homotopy-barycenter+photography-proof1}
	we infer
	\begin{equation*}
		\mathrm{dist}_g\big(\mathcal{B}(\mathcal{P}_{\varepsilon,m}(p)),p\big)
		\le 2C_{\partial M}r^* \leq \frac{1}{2}\mathrm{inj}(\partial M).
	\end{equation*}
	As a consequence, the map
	$F\colon [0,1] \times \partial M \to \partial M$,
	given by
	\begin{equation*}
		F(t,p)\coloneqq
		\mathrm{exp}^{\partial M}_p\Big(t
			\big(\mathrm{exp}^{\partial M}_p\big)^{-1}
		\big(\mathcal{B}(\mathcal{P}_{\varepsilon,m}(p))\big)\Big),
	\end{equation*}
	is well-defined,
	continuous and gives a homotopy equivalence between
	$\mathcal{B}\circ \mathcal{P}_{\varepsilon,m}$ and
	the identity map in $\partial M$.
\end{proof}

Finally, we are ready to prove Theorem~\ref{theorem:ACH-Neumann}.
\begin{proof}[Proof of Theorem~\ref{theorem:ACH-Neumann}]
	For every $m,\varepsilon > 0$, 
	the functional $\mathcal{E}_{\varepsilon,m}$
	is clearly bounded below and,
	by Lemma~\ref{lem:E_eps_C2-PS},
	it is of class $C^1$ and satisfies the Palais-Smale condition.

	Let us choose $m^*$ as in 
	Lemma~\ref{lem:homotopy-barycenter+photography}
	and, for every $m \in (0,m^*)$,
	let $\varepsilon_{m}$ be equal to $\varepsilon^*$ of the same lemma
	and set $c_m = \sigma I_M(m) + \theta m$.
	By Proposition~\ref{prop:photography-continuity}
	and Remark~\ref{rem:def-r1}, 
	for every $\varepsilon \in (0,\varepsilon_m)$
	both
	$\mathcal{P}_{\varepsilon,m}\colon \partial M \to
	\mathcal{E}_{\varepsilon,m}^{c_m}$
	and $\mathcal{B}\colon \mathcal{E}_{\varepsilon,m}^{c_m} \to \partial M$
	are continuous
	and, by Lemma~\ref{lem:homotopy-barycenter+photography},
	their composition is homotopic to the identity.
	Then, Theorem~\ref{theorem:ACH-Neumann}
	directly follows from Theorem~\ref{theorem:photography}.
	In particular, there exist at least $\cat(\partial M)$
	critical points of $\mathcal{E}_{\varepsilon,m}$
	in $\mathcal{E}_{\varepsilon,m}^{c_m}$ and,
	recalling that $\cat(\partial M) > 1$
	and $\mathcal{H}_{m}$ is a contractible set
	(see Remark~\ref{rem:cat-partial-M}),
	there exist at least one critical point with energy larger than $c_m$.
	By Lemma~\ref{lem:E_eps_C2-PS}, 
	these critical points are solutions of~\eqref{eq:PDE_Neumann}.
	Moreover, if all the critical points of $\mathcal{E}_{\varepsilon,m}$
	are non-degenerate, then 
	Theorem~\ref{theorem:photography} and Lemma~\ref{lem:E_eps_C2-PS}
	ensure that
	\eqref{eq:PDE_Neumann} has at least $\mathscr{P}_1(\partial M)$
	solutions with energy less than $c_{m}$ and $\mathscr{P}_1(\partial M)-1$
	solutions with energy larger than $c_{m}$,
	and we are done.
\end{proof}

\section{Proof of Theorem \ref{theorem:ACH-Dirichet}}%
\label{sec:dirichlet}
The proof of Theorem~\ref{theorem:ACH-Dirichet} follows
essentially the same structure of the proof of Theorem~\ref{theorem:ACH-Neumann},
by applying again Theorem~\ref{theorem:photography}.
Relying on some $\Gamma$--convergence results for 
$\overline{\CE}_{\eps,m}$ (\cite{owen-rubinstein-sternberg,MR1382825}),
we can construct the photography map
$\mathcal{P}_{\varepsilon,m}\colon M \to \overline{\mathcal{E}}_{\varepsilon,m}^c$,
where the value of $c$ can be estimated 
exploiting the convergence of the isoperimetric problem
on $M$ that takes into account also
the boundary $\partial M$ to the standard Euclidean case 
as $m$ goes to $0$.

\subsection{The isoperimetric problem}

We define the functional $\overline{\mathcal{E}}_0\colon L^1(M,\mathbb{R}) \to \mathbb{R}$
as
\begin{equation*}
	\overline{\mathcal{E}}_0(u)\coloneqq \begin{cases}
		\sigma\CH^{n-1}(\partial^* \Omega), 
		&\mbox{ if } u=\mathbf{1}_{\Omega},\\
		+\infty, &\mbox{ otherwise},
	\end{cases}
\end{equation*}
where $\sigma \in \mathbb{R}$ is given by~\eqref{eq:def-sigma}.
Let $\overline{I}\colon ]0,\vol(M)] \to \mathbb{R}$ be the following function:
\begin{equation}
	\label{IP_boundary}
	\overline{I}_{M}(m)\coloneqq 
	\inf\left\{ \mathcal{H}^{n-1}(\partial^*\Omega):
		\Omega \in \mathcal{C}_g(M)
		\text{ and }
	\int_M \mathbf{1}_{\Omega} \de v_g = m\right\}.
\end{equation}
The quantity $\overline{I}_M$ is the isoperimetric profile of $M$
that takes also into account the boundary $\partial M$.
For small volumes, the function $\overline{I}_{M}$
converges to its Euclidean counterpart,
which is $I_{\mathbb{R}^n}(m) = c_n m^{\frac{n-1}{n}}$,
where $c_n$ is the Euclidean isoperimetric constant.
More formally, we have the following result.
\begin{lemma}
	\label{lem:lineIM-to-IR}
	The following equality holds:
	\begin{equation}
		\label{eq:lineIM-to-IR}
		\lim_{m \to 0^+}
		\frac{\overline{I}_M(m)}{I_{\mathbb{R}^n}(m)}
		=
		\lim_{m \to 0^+}
		\frac{\overline{I}_M(m)}{c_n m^{\frac{n-1}{n}}}
		= 1.
	\end{equation}
\end{lemma}
\begin{proof}
	For every $p \in \mathrm{int}(M)$,
	if $m$ is sufficiently small we have that the geodesic ball
	centered at $p$ and with volume $m$ does not intersect the boundary.
	Setting $u_{0,p}^m \in L^1(M,\mathbb{R})$ the characteristic function of this ball,
	by \cite[Lemma 3.10]{MR2529468} we have
	\[
		\overline{\mathcal{E}}_{0}(u_{0,p}^m) = \sigma c_n m^{\frac{n-1}{n}} + o(m^{\frac{n-1}{n}}).
	\]
	As a consequence, if $m$ is sufficiently small we have
	$\overline{I}_{M}(m) \le c_n m^{\frac{n-1}{n}} + o(m^{\frac{n-1}{n}})$,
	so we obtain 
	\[
		\limsup_{m\to 0}
		\frac{\overline{I}_M(m)}{I_{\mathbb{R}^n}(m)} \le 1.
	\]
	So it remains to prove that
	\[
		\liminf_{m\to 0}
		\frac{\overline{I}_M(m)}{I_{\mathbb{R}^n}(m)} \ge 1.
	\]
	Arguing by contradiction,
	there exists a sequence $(m_k)_{k \in \mathbb{N}}$ which goes to $0$,
	a sequence of isoperimetric sets
	$(\Omega_k)_{k \in \mathbb{N}}\subset \mathcal{C}_g(M)$
	such that $\vol(\Omega_k) = m_k$ for every $k\in\mathbb{N}$
	and $\alpha \in (0,1)$ such that
	\begin{equation*}
		\overline{I}_{M}(m_k)
		=\mathcal{H}^{n-1}(\partial^*\Omega)
		\le \alpha c_n m_k^{\frac{n-1}{n}},
		\qquad \forall k \in \mathbb{N}.
	\end{equation*}
	Now, we rely on Theorem~\ref{theorem:antonelli2022}
	to obtain the desired contradiction.
	We consider the sequence of manifolds of bounded geometry
	given by $(M_k,g_k,\mathcal{H}_{k}^n) = (M,m_{k}^{-1/n}g,\mathcal{H}_{k}^n)$.
	We remind that this sequence is such that
	$\mathcal{H}_{k}^n(\Omega_k) = 1$ for every $k \in \mathbb{N}$,
	where $\mathcal{H}_{k}^{n}$ denotes the $n$--dimensional Hausdorff measure 
	on $(M_k,g_k)$.
	Moreover, the rescaling factor is such that 
	\begin{equation*}
		\mathcal{H}^{n-1}_k(\partial^*\Omega_k)\le \alpha c_n,
		\qquad \forall k \in \mathbb{N}.
	\end{equation*}
	This means that $(\Omega_k)_{k \in \mathbb{N}}$
	is a sequence of sets with uniformly bounded volumes and perimeters
	and we can apply Theorem~\ref{theorem:antonelli2022}.
	Hence, there exists a non-decreasing sequence
	$(N_k)_{k \in \mathbb{N}}\subset \mathbb{N}$
	and for every $k$ we have a set of points
	$\{p_{k,i} \in M_k, i = 1,\dots,N_k\}$
	such that
	$(M_k,\dist_{g_k},\mathcal{H}^n_k,p_{k,i})$
	converges in the Gromov-Hausdorff topology 
	as $k\to \infty$
	to either the Euclidean semi-space
	$(\mathbb{R}^n_+,\de,\mathcal{H}^n,0)$
	or the Euclidean space $(\mathbb{R}^n,\de,\mathcal{H}^n,0)$.
	For every $k\in \mathbb{N}$ we have
	a set of pairwise disjoint sets
	$\Omega_{k}^i\subset \Omega_k \subset M$, for $i = 1,\dots,N_k$,
	such that, setting $\bar{N} = \lim_{k\to \infty}N_k$,
	we have 
	$\lim_{k\to\infty}\Omega_{k}^i = \Omega_{\infty}^i\subset \mathbb{R}^n$
	for every $i = 1,\dots,\bar{N}$
	and 
	\begin{equation*}
		\sum_{i = 1}^{\bar{N}} \mathcal{H}^{n}(\Omega_{\infty}^i) 
		= \lim_{k \to \infty}\mathcal{H}^n_k(\Omega_k)
		= 1.
	\end{equation*}
	Moreover, we obtain the following inequality
	\begin{equation}
		\label{eq:lineIM-to-IR-proof1}
		\sum_{i = 1}^{\bar{N}}\mathcal{H}^{n-1}(\partial^* \Omega_{\infty}^i)
		\le \liminf_{k \to \infty}\mathcal{H}_{k}^{n-1}(\partial^*\Omega_{k})
		< c_n,
	\end{equation}
	and every $\Omega^{i}_{\infty}$
	is an isoperimetric set in $\mathbb{R}^n_+$
	or in $\mathbb{R}^n$.
	As a consequence, for every $i$ 
	we know that $\Omega^{i}_{\infty}$
	is a bounded set,
	so we can see them as a family of pairwise 
	disjoint sets in $\mathbb{R}^n$.
	This means that the set $\Omega_{\infty} = \cup_{i = 1}^{\bar{N}}\Omega_{\infty}^i$
	is a subset of $\mathbb{R}^n$ with volume $1$
	and, by~\eqref{eq:lineIM-to-IR-proof1},
	its perimeter is strictly less than $c_n$,
	which is the perimeter of the ball with the same volume,
	obtaining the desired contradiction.
\end{proof}

\subsection{Gamma convergence}

\begin{proposition}[cf. Proposition A and Proposition B of \cite{MR1382825}]
	\label{prop:Gamma-conv-Dirichlet}
	Assume that \ref{asu_nondegenerate} and \ref{asu_coercive} hold.
	Then, the following statements hold:
	\begin{enumerate}[i)]
		\item \textbf{Lim-inf}:
			If $(\eps_k)_{k \in \N}\subset ]0,+\infty[$
			is such that $\eps_k \to 0^+$
			and $(u_{\eps_k})_{k \in \N} \subset H^1(M,\R)$
			is such that $u_{\eps_k} \to u_0$ in $L^1(M,\R)$,
			then $\liminf_{k \to \infty}
			\overline{\mathcal{E}}_{\varepsilon_k}(u_{\eps_k})
			\ge \overline{\mathcal{E}}_0(u_0)$;
		\item \textbf{Lim-sup}:
			For any $u_0 \in L^1(M,\R)$ such that $u_0 = \mathbf{1}_{\Omega}$
			for some finite perimeter measurable set $\Omega \subset M$ 
			and for every sequence
			$(\eps_k)_{k \in \N}\subset ]0,+\infty[$ such that $\eps_k \to 0^+$,
			there exists $(u_{\eps_k})_{k \in \N} \subset H^1(M,\R)$
			such that $u_{\eps_k} \to u_0$ in $L^1(M,\R)$,
			\[
				\int_{M} u_{\varepsilon_k} \de v_g = \int_M u_0\de v_g
				\quad\text{and}\quad
				\limsup_{k \to \infty}
				\overline{\mathcal{E}}_{\varepsilon_k}(u_{\eps_k})
				\le \overline{\mathcal{E}}_0(u_0).
			\]
	\end{enumerate}

\end{proposition}
\begin{remark}
	\label{rem:recovery-seq-Dirichlet}
	Noticing that the compactness result
	ensured by Theorem~\ref{theorem:recovery-sequence}
	refers to the functional $\mathcal{E}_{\varepsilon}\colon H^1(M,\mathbb{R}) \to \mathbb{R}$,
	as for the case of Neumann boundary condition we have
	that if $(\varepsilon_k)_{k \in \mathbb{N}}\subset ]0,+\infty[$
	is a sequence such that $\varepsilon_k \to 0^+$
	and if a sequence
	$(u_{\varepsilon_k})_{k \in \mathbb{N}}\subset \mathcal{H}_{m,0}$
	satisfies
	$\overline{\mathcal{E}}_{\varepsilon_k}(u_{\varepsilon_k})\le E^*$
	for some constant $E^*$,
	then, up to subsequences, $(u_{\varepsilon_k})_{k \in \mathbb{N}}$
	converges to a function $u_{0} \in L^1(M,\mathbb{R})$.
	By applying the lim-inf property 
	of Proposition~\ref{prop:Gamma-conv-Dirichlet},
	we have that $\overline{\mathcal{E}}_0(u_0)\le E^*$,
	hence there exists a measurable and finite perimeter set $\Omega\subset M$
	such that $u_0 = \mathbf{1}_{\Omega}$
	and $\vol(\Omega) = \int_{M}u_0 \de v_g = m$.
\end{remark}

\subsection{Photography map}
\label{sec:photo-Dirichlet}
Since now we are dealing with the Dirichlet boundary condition,
namely $u\equiv 0$ on $\partial M$,
the photography map is defined by utilizing a \emph{boundary layer}. 
The main idea is to associate to each point of the manifold
the $\varepsilon$--approximation of the characteristic function of
the geodesic ball with prescribed volume centered in the point.
However, when the point is near the boundary,
such a ball may intersect with the boundary
and the $\varepsilon$--approximation does not satisfy the boundary condition.
To avoid this problem, we construct a boundary layer map
$\mathcal{L}\colon M \to M$
that slightly moves inside the points
that lie on tubular neighbourhood of $\partial M$.
In this way, the geodesic ball of prescribed volume $m$
centered at $\mathcal{L}(p)$ is in the interior of $M$ for every $p\in M$,
provided $m$ sufficiently small.

Let us proceed with a formal construction. 
Since the boundary of $M$ is smooth and compact,
there exists $\delta_M > 0$ such that the map
\[
	\partial M \times [0,\delta_M] \ni
	(Q,t) \mapsto 
	\exp_{Q}(t N_{\partial M})
\]
provides a coordinate system in a neighbourhood of $\partial M$,
where we recall that $N_{\partial M}$
stands for the unit interior normal vector field along $\partial M$.
Now, for any $p \in M$ such that $\dist(p,\partial M)\le \delta_M$,
we denote by $(Q_p,t_p)$ the unique element in $\partial M \times [0,\delta_M]$
such that
\[
	p = \exp_{Q_p}(t_p N_{\partial M}).
\]
In other words, $t_p$ is the distance of $p$ to the boundary,
while $Q_p$ is its projection on it.
Now, let us choose a $C^\infty$ function $h\colon [0,\delta_M] \to [0,\delta_M]$
that is strictly increasing (hence invertible),
$h(0) = \delta_M/2$, $h(\delta_M) = \delta_M$
and $h'(\delta_M) = 1$.
Our boundary-layer function
$\mathcal{L}\colon M \to M$
is defined as follows:
\begin{equation}
	\label{eq:def-boundary-layer}
	\mathcal{L}(p) \coloneqq
	\begin{dcases}
		\exp_{Q_p}\big(h(t_p) N_{\partial M}\big),
		&		\mbox{if } \dist(p,\partial M) \in [0,\delta_M], \\
		p, &		\mbox{if } \dist(p,\partial M) \ge \delta_M.
	\end{dcases}
\end{equation}
Notice that $\mathcal{L}$ is a $C^1$--map 
homotopic to the identity.

Since $\dist(\mathcal{L}(p),\partial(M))\ge \delta_M/2$
for every $p \in M$ and $M$ is compact,
there exists a sufficiently small volume,
say $m_0 > 0$, such that 
for any $p \in M$ and $m \in (0,m_0)$
the geodesic ball centered at $\mathcal{L}(p)$
and with volume $m$ doesn't intersect the boundary
of the manifold.
More formally, denoting by $r_{q,m} > 0$ the radius 
of the geodesic ball centered at $q\in M$ with volume $m$,
hence 
$\int_{B(q,r_{q,m})}1 \de v_g = m$,
we have
\begin{equation}
	\label{eq:def-m_0}
	B(\mathcal{L}(p),r_{\mathcal{L}(p),m}) \cap \partial M = \emptyset,
	\qquad \forall p \in M, m \in (0,m_0).
\end{equation}
For any $p \in M$ and $m \in (0,m_0)$
let us denote by $u_{0,\mathcal{L}(p)}^m$ the 
characteristic function of 
$B(\mathcal{L}(p),r_{\mathcal{L}(p),m})$.
Moreover, for every $\varepsilon > 0$
let $u_{\varepsilon,\mathcal{L}(p)}^m \in \mathcal{H}_{m}$
be $\varepsilon$--approximation of $u_{0,\mathcal{L}(p)}^m$
given by the {lim-sup} property of 
Proposition~\ref{prop:Gamma-conv-Dirichlet}.
Since $M$ is compact and
$\supp\, u_{0,\mathcal{L}(p)}^{m} \subset\joinrel\subset \mathrm{int}(M)$
for every $p \in M$ and $m \in (0,m_0)$,
there exists $\varepsilon_0 > 0$
such that for every $\varepsilon \in ]0,\varepsilon_0[$
we have 
$u_{\varepsilon,\mathcal{L}(p)}^m \in \mathcal{H}_{m,0}$,
hence $u_{\varepsilon,\mathcal{L}(p)}\equiv 0$ on $\partial M$
(see Remark~\ref{rem:epsilon-approximation} for 
the details about the construction of the $\varepsilon$--approximation).

\begin{definition}
	\label{def:photography-Dirichlet}
	For every $m \in (0,m_0)$ and $\varepsilon \in (0,\varepsilon_0)$,
	we define the photography map
	$\mathcal{P}_{\varepsilon,m}\colon M \to \mathcal{H}_{m,0}$
	as follows:
	\[
		\mathcal{P}_{\varepsilon,m}(p) = u_{\varepsilon,\mathcal{L}(p)}^m.
	\]
\end{definition}
See Figure \ref{FIG_photography} for a representation of Definition \ref{def:photography-Dirichlet}.
\begin{figure}
    \centering
    \begin{tikzpicture}
    \node[anchor=south west] (image) at (0,0) {\includegraphics[width=\textwidth]{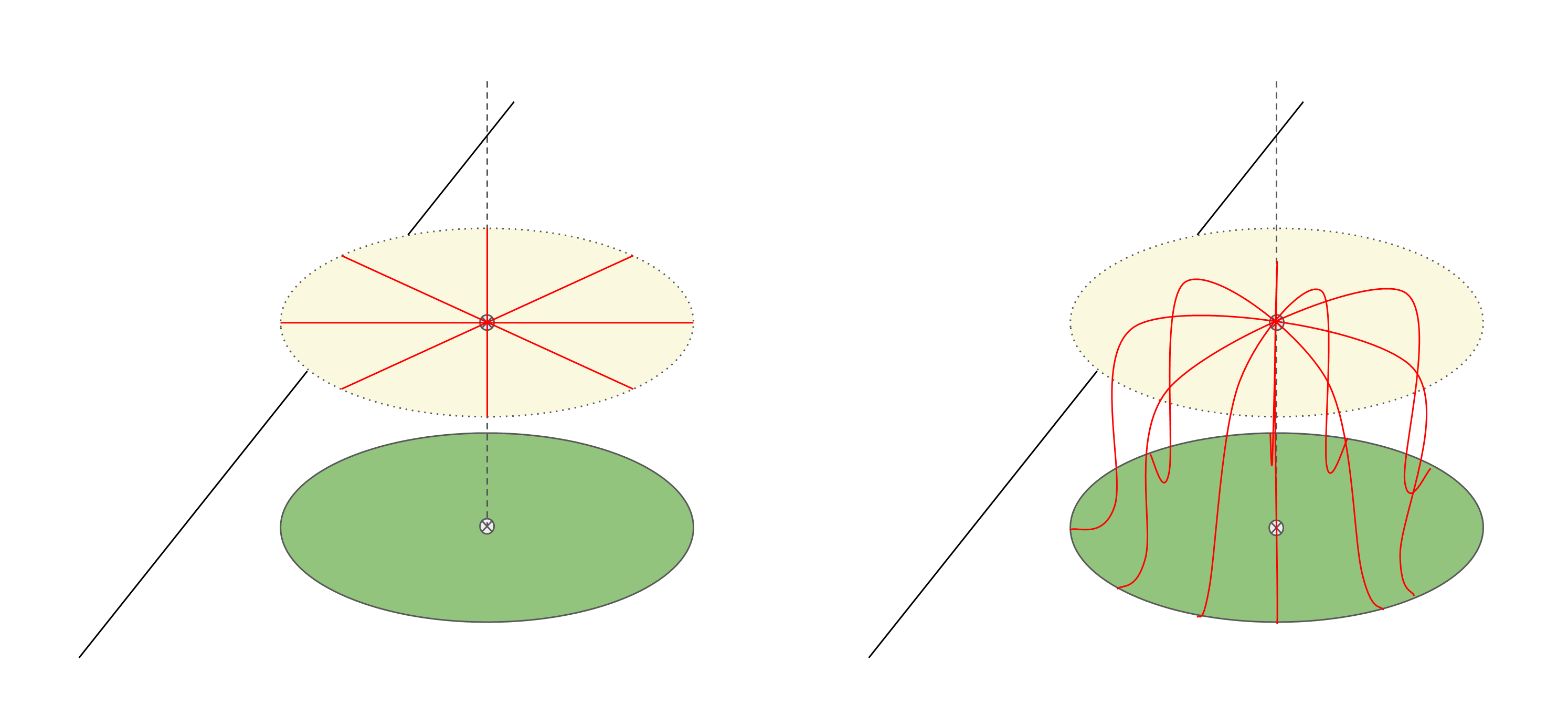}};
    \put(45,60){$p$};
    \put(140,45){$\mathcal{L}(p)$};
    \put(110,150){$\partial M$};
    \put(50,20){$M$};
    \put(82,12){$B(\mathcal{L}(p),r_{\mathcal{L}(p),m})$};
    \put(195,50){$0$};
    \put(195,100){$1$};
    \put(150,135){$u^m_{0,\mathcal{L}(p)}$};
    \put(258,60){$p$};
    \put(353,45){$\mathcal{L}(p)$};
    \put(321,150){$\partial M$};
    \put(263,20){$M$};
    \put(295,12){$B(\mathcal{L}(p),r_{\mathcal{L}(p),m})$};
    \put(408,50){$0$};
    \put(408,100){$1$};
    \put(363,135){$u^m_{\varepsilon,\mathcal{L}(p)}$};

    \draw [->](2,2) -- (4.55,1.8);
    \draw [->](9.4,2) -- (11.98,1.8);
    \foreach \Point in {
		(2,2),
            (4.72,1.8),
            (9.4,2),
            (12.12,1.8)
	}{
		\node at \Point {\textbullet};
	}
    \end{tikzpicture}
    \caption{
    The map introduced in Definition \ref{def:photography-Dirichlet} turns the indicator function $u_{0,\mathcal{L}(p)}^m$ into the smooth approximation $u_{\eps,\mathcal{L}(p)}^m$ given by $\Gamma$-convergence (Proposition \ref{prop:Gamma-conv-Dirichlet}). The functions are supported on the balls $B(\mathcal{L}(p),r_{\mathcal{L}(p),m})$, whose center $\mathcal{L}(p)$ is obtained by sending a point $p$ away from the boundary through the boundary layer map $\mathcal{L}$.}
    \label{FIG_photography}
\end{figure}
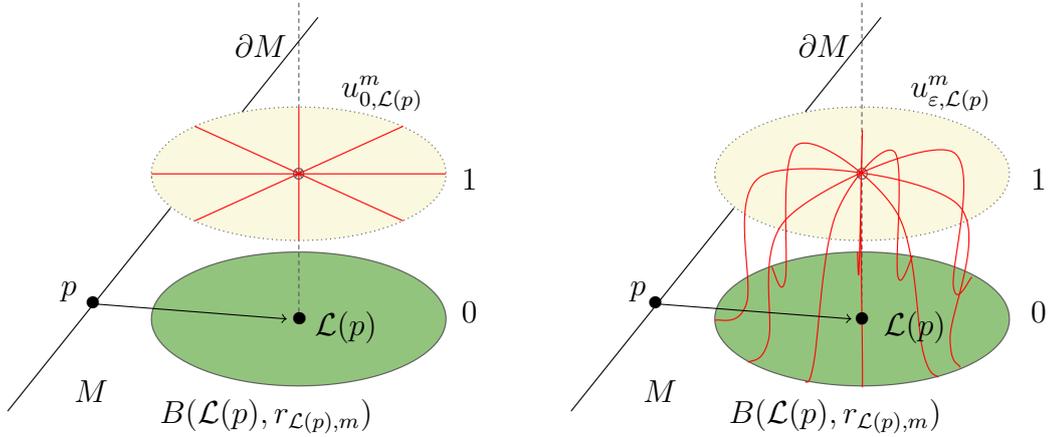

\begin{proposition}
	\label{prop:photo-sublevel-Dirichlet}
	Assume that \ref{asu_nondegenerate} and \ref{asu_coercive} hold. There exists a function $\tau\colon \mathbb{R}^+ \to \mathbb{R}$
	such that 
	\[
		\lim_{m \to 0^+} \frac{\tau(m)}{m^{\frac{n-1}{n}}} = 0
	\]
	and there exists $m_1= m_1(M,g,W,\tau) \in (0,m_0)$
	such that for every $m \in (0,m_1)$
	there exists
	$\varepsilon_1 = \varepsilon_1(M,g,W,m) \in (0,\varepsilon_0)$
	such that for every $\varepsilon \in (0,\varepsilon_1)$
	we have
	\begin{equation}
		\label{eq:photo-sublevel-Dirichet}
		\overline{\mathcal{E}}_{\varepsilon,m}(\mathcal{P}_{\varepsilon,m}(p))
		\le \sigma \overline{I}_{M}(m) + \tau(m),
		\qquad \forall p \in M.
	\end{equation}
\end{proposition}

\begin{proof}
	By applying some standard results
	(see, e.g., \cite[Lemma 3.10]{MR2529468}),
	since $u_{0,\mathcal{L}(p)}^m$
	has compact support in the interior of $M$
	for every $m \in (0,m_0)$
	we have
	\begin{equation}
		\label{eq:perimeter-of-geo-balls}
		\overline{\mathcal{E}}_{0}(u_{0,\mathcal{L}(p)}^m)
		=
		\sigma\left(
			c_n m^{\frac{n-1}{n}}
			- \gamma_n \mathrm{Sc}_g(\mathcal{L}(p))m^{\frac{n+1}{n}}
		\right)
		+ \mathcal{O}(m^{\frac{n+3}{n}}),
		\qquad \text{as } m \to 0^+,
	\end{equation}
	where
	$c_n$ is the Euclidean isoperimetric constant,
	$\gamma_n$ is a constant which depends only on the dimension
	of the manifold and
	$\mathrm{Sc}_g(\mathcal{L}(p))$
	denotes the scalar curvature of the metric tensor $g$
	at the point $\mathcal{L}(p)$.
	Since $M$ is compact,
	there exists a constant $\omega > 0$ 
	and $m_1 \in (0,m_0)$ such that
	for every $m\in (0,m_1)$ we have
	\begin{equation*}
		\overline{\mathcal{E}}_{0}(u_{0,\mathcal{L}(p)}^m)
		< \sigma c_n m^{\frac{n-1}{n}} + \omega m^{\frac{n+1}{n}},
		\qquad \forall p \in M.
	\end{equation*}
	Since the last inequality is strict and 
	using again the compactness of $M$,
	by the lim-sup property of Proposition~\ref{prop:Gamma-conv-Dirichlet}
	there exists $\varepsilon_1 = \varepsilon_1(M,g,m) \in (0,\varepsilon_0)$
	such that for every $\varepsilon \in (0,\varepsilon_1)$
	we have
	\begin{equation}
		\label{eq:level-photo-Dirichlet-proof}
		\overline{\mathcal{E}}_{\varepsilon,m}(u_{\varepsilon,\mathcal{L}(p)}^m)
		\le\sigma  c_n m^{\frac{n-1}{n}} + \omega m^{\frac{n+1}{n}},
		\qquad \forall p \in M.
	\end{equation}
	By~\eqref{eq:lineIM-to-IR},
	the function $\tau_0\colon \mathbb{R}^+ \to \mathbb{R}$
	given by 
	$\tau_0(m) = c_n m^{\frac{n-1}{n}}- \overline{I}_{M}(m)$
	is a $o(m^{\frac{n-1}{n}})$ as $m$ goes to $0$.
	As a consequence, setting $\tau(m) = \sigma \tau_0(m) + \omega m^{\frac{n+1}{n}}$,
	\eqref{eq:level-photo-Dirichlet-proof}
        is equivalent to~\eqref{eq:photo-sublevel-Dirichet}.
\end{proof}

Since $\mathcal{L}\colon M \to M$ is a continuous function,
one can employ the same construction of the proof of
Proposition~\ref{prop:photography-continuity}
to obtain the continuity of the photography map,
which is stated by the next result.
\begin{proposition}
	\label{prop:continuity-photo-Dirichlet}	
	Assume that \ref{asu_nondegenerate} and \ref{asu_coercive} hold. Let $m_1> 0$ and $\varepsilon_1 > 0$ be defined as in 
	Proposition~\ref{prop:photo-sublevel-Dirichlet}.
	For every $m \in (0,m_1)$ and $\varepsilon \in (0,\varepsilon_1)$
	the map $\mathcal{P}_{\varepsilon,m}\colon M \to \mathcal{H}_{m,0}$
	is continuous.
\end{proposition}
\begin{proof}
	Cf. Proposition~\ref{prop:photography-continuity}.
\end{proof}

\subsection{Barycenter map}
The construction of the barycenter map is analogous to
the case of Neumann boundary condition.
In particular, we will use again
the function $\beta^*\colon L^1(M,\mathbb{R}) \to \mathbb{R}^{\tilde{n}}$
defined in~\eqref{eq:def-beta*},
which gives the ``extrinsic'' center of mass of a function
in the Euclidean space $\mathbb{R}^{\tilde{n}}$ where the manifold $M$
is isometrically embedded.
In this case, we will compose this map with 
the nearest point projection map
$\pi_M\colon \mathbb{R}^{\tilde{n}} \to M$.
We need to prove that, when $m$ and $\varepsilon$
are sufficiently small,
$\beta^*(u)$ is near the manifold for any
function $u$ that belongs to the sublevel of
$\overline{\mathcal{E}}_{\varepsilon,m}$
that contains the image of the photography map,
namely for every 
$u \in \overline{\mathcal{E}}_{\varepsilon,m}^c$
with $c = \sigma\overline{I}_{M}(m) + \tau(m)$
(see Proposition~\ref{prop:photo-sublevel-Dirichlet}).
To obtain the last result, we rely on the ``concentration property'',
namely that as $m$ goes to $0$
the support of every function in 
$\overline{\mathcal{E}}_{\varepsilon,m}^c$
is inside a small ball,
up to a negligible part.

\begin{proposition}
	\label{prop:almost-isoperimetric-in-a-small-ball-Dirichlet}
	There exists $\mu = \mu(M,g)> 0$ such that the following property holds.
	For every almost isoperimetric sequence
	$(u_k)_{k\in \mathbb{N}}\subset  L^1(M,\mathbb{R})$
	with volumes $m_k = \int_{M} u_{k}\de v_g \to0$,
	i.e., 
	\begin{equation}
		\label{eq:almost-isoperimetric-in-a-small-ball-Dirichlet-HP}
		\lim_{k \to \infty}
		\frac{\overline{\mathcal{E}}_{0}(u_k)}{\sigma \overline{I}_M(m_k)} = 1,
	\end{equation}
	there exists a sequence $(p_k)_{k \in \mathbb{N}} \subset M$
	such that
	\begin{equation}
		\label{eq:almost-isoperimetric-in-a-small-ball-Dirichlet}
		\lim_{k \to +\infty}
		\frac{1}{m_k}
		\left(
			\int_{M\setminus B(p_k,\mu m_{k}^{1/n})} u_{k} \de v_g
		\right)
		= 0.
	\end{equation}
\end{proposition}
\begin{proof}
	The proof is analogous to the one of
	Proposition~\ref{prop:almost-isoperimetric-in-a-small-ball},
	and it is based on the compactness result ensured by 
	Theorem~\ref{theorem:antonelli2022}.

	Let $(\Omega_k)_{k \in \mathbb{N}}\subset \mathcal{C}_g$
	such that $u_k = \mathbf{1}_{\Omega_k}$
	for every $k$.
	Let $(X_k,g_k) = (M,m^{-1/k}g)$,
	so that $\mathcal{H}^n_k(\Omega_k) = 1$ for every $k$,
	where $\mathcal{H}^{n}_k$ denotes the $n$--dimensional Hausdorff measure on $(M_k,g_k)$.
	By~\eqref{eq:lineIM-to-IR}
	and~\eqref{eq:almost-isoperimetric-in-a-small-ball-Dirichlet-HP},
	we have also that 
	\begin{equation*}
		\lim_{k\to\infty}P(\Omega_k,X_k)
		= \lim_{k\to\infty}\mathcal{H}_k^{n-1}(\partial^*\Omega_k)
		= I_{\mathbb{R}^n}(1).
	\end{equation*}
	This means that, seen as subsets of $(M_k,g_k)$,
	the sequence $(\Omega_k)_{k\in \mathbb{N}}$
	has uniformly bounded volumes and perimeters and we can apply
	Theorem~\ref{theorem:antonelli2022}
        to obtain
	a number $\overline{N} \in \mathbb{N}\cup\{+\infty\}$,
	a family of pointed $\mathrm{RCD}(\kappa,n)$
	spaces $(M^i_\infty,\de^i_{\infty},\mathcal{H}^{n}_{\infty},p^i)$,
	for $i = 1,\dots,\overline{N}$,
	and $\Omega^i_{\infty}\subset M^i_{\infty}$
	such that 
	\begin{equation}
		\label{eq:almost-isoperimetric-in-a-small-ball-Dirichlet-proof1}
		\sum_{i = 1}^{\overline{N}}\mathcal{H}^n_{\infty}(\Omega^i_{\infty})
		= \lim_{k\to\infty}\mathcal{H}^{n}_k(\Omega_k)
		= 1,
	\end{equation}
	and
	\begin{equation}
		\label{eq:almost-isoperimetric-in-a-small-ball-Dirichlet-proof2}
		\sum_{i = 1}^{\overline{N}}\mathcal{H}^{n-1}_{\infty}(\partial^*\Omega^i_\infty)
		\le \liminf_{k \to \infty}\mathcal{H}^{n-1}_k(\partial^{*}\Omega_k)
		= I_{\mathbb{R}^n}(1).
	\end{equation}
	By construction, 
	each $(M^i_\infty,\de^i_{\infty},\mathcal{H}^{n}_{\infty},p^i)$
	is either the Euclidean semi-space
	$(\mathbb{R}^{n}_+,\de,\mathcal{H}^{n},0)$
	or the Euclidean space 
	$(\mathbb{R}^{n},\de,\mathcal{H}^{n},0)$.
	Since $I_{\mathbb{R}^n}(m) = c_nm^{\frac{n-1}{n}}$
	is a strictly concave function and $I_{\mathbb{R}^n}(0) = 0$,
	it is strictly subadditive and 
	by using a contradiction argument
	\eqref{eq:almost-isoperimetric-in-a-small-ball-Dirichlet-proof1}
	and~\eqref{eq:almost-isoperimetric-in-a-small-ball-Dirichlet-proof2}
	show that $\overline{N} = 1$.
	Using again~\eqref{eq:almost-isoperimetric-in-a-small-ball-Dirichlet-proof1},
	this implies the existence of $\mu > 0$
	and of a sequence of points $(p_k)_{k \in \mathbb{N}}\subset M$
	such that~\eqref{eq:almost-isoperimetric-in-a-small-ball-Dirichlet}
	holds.
\end{proof}

\begin{lemma}
	\label{lem:in-a-small-sphere-Dirichlet}
	Assume that \ref{asu_nondegenerate} and \ref{asu_coercive} hold
        and let $\tau\colon \mathbb{R}^+ \to \mathbb{R}$
	be given by Proposition~\ref{prop:photo-sublevel-Dirichlet}.
	For every $\alpha \in (0,1)$,
	there exists $m_{\alpha} = m_{\alpha}(M,g,W,\theta,\alpha) > 0$
	such that for every $m \in (0,m_{\alpha})$
	there exists $\varepsilon_{\alpha} = \varepsilon_{\alpha}(M,g,W,\theta,\alpha,m) > 0$
	such that for every $\varepsilon \in (0,\varepsilon_{\alpha})$
	and for any 
	$u \in \mathcal{E}_{\varepsilon,m}^{\sigma \overline{I}_M(m) + \tau(m)}$
	there exists a point $p_u \in M$
	such that
	\begin{equation}
		\label{eq:mass-in-a-small-sphere-Dirichlet}
		\int_{M\setminus B(p_u,\mu m^{1/n})} |u| \de v_g
		\le \alpha m.
	\end{equation}
\end{lemma}
\begin{proof}
	Relying on the concentration result ensured by 
	Proposition~\ref{prop:almost-isoperimetric-in-a-small-ball-Dirichlet},
	the proof is analogous to the one of Lemma~\ref{lem:in-a-small-semisphere}.
	For the sake of clarity, we report here the main steps.

	Arguing by contradiction, there exists a sequence $(m_i)_{i \in \mathbb{N}}$
	that goes to $0$ and for every $i \in \mathbb{N}$
	there exists two sequences $(\varepsilon_{i,j})_{j \in\mathbb{N}}$
	and $(u_{i,j})_{j \in \mathbb{N}}\subset \overline{\mathcal{E}}_{\varepsilon_{i,j},m_i}^{c_i}$,
	with $c_i = \sigma \overline{I}_{M}(m_i)+\tau(m_i)$
	such that $\varepsilon_{i,j}\to 0^+$ as $j \to \infty$
	and 
	\begin{equation}
		\label{eq:in-a-small-sphere-proof1}
		\int_{M\setminus B_g(p,\mu m_i^{1/n})} |u_{i,j}| \de v_g
		> \alpha m_i,
		\qquad \forall p \in M,\, \forall j \in \mathbb{N}.
	\end{equation}
	Since $\overline{\mathcal{E}}_{\varepsilon_{i,j},m_i}(u_{i,j})\le c_i$
	for every $j$, by Theorem~\ref{theorem:recovery-sequence}
	we have that for every fixed $i \in\mathbb{N}$
	there exists a characteristic function $u_{0,i}\in L^1(M,\mathbb{R})$
	such that $u_{i,j}$ converges to it, up to subsequences.
	Therefore, for every $i \in \mathbb{N}$ there exists $j_i$
	sufficiently large such that
	\begin{equation}
		\label{eq:in-a-small-sphere-proof2}
		\int_{M}\left(|u_{i,j_i}-u_{0,i}|\right)\de v_g \le \frac{\alpha}{4}m_i.
	\end{equation}
	Moreover, by Remark~\ref{rem:recovery-seq-Dirichlet}
	we have 
	\[
		\overline{I}_{M}(m_i)
		\le\frac{\overline{\mathcal{E}}_{0}(u_{0,i})}{\sigma}
		\le \overline{I}_{M}(m_i) + \tau(m_i).
	\]
	Since $\tau(m) = o(m^{\frac{n-1}{n}})$ and~\eqref{eq:lineIM-to-IR}
	holds, the last chain of inequality implies that
	\begin{equation*}
		\lim_{i \to \infty}
		\frac{\overline{\mathcal{E}}_{0}(u_{0,i})}{\sigma \overline{I}_M(m_i)} = 1,
	\end{equation*}
	so we can apply Proposition~\ref{prop:almost-isoperimetric-in-a-small-ball-Dirichlet}
	and obtain a sequence of points $(p_i)_{i \in \mathbb{N}}\subset M$
	such that 
	\begin{equation}
		\label{eq:in-a-small-sphere-proof3}
		\lim_{i \to +\infty}
		\frac{1}{m_i}
		\left(
			\int_{M\setminus B(p_i,\mu m_{i}^{1/n})} u_{0,i} \de v_g
		\right)
		= 0.
	\end{equation}
	Utilizing the same reasoning employed in the concluding section
	of the proof for Lemma~\ref{lem:in-a-small-semisphere},
	\eqref{eq:in-a-small-sphere-proof2} and~\eqref{eq:in-a-small-sphere-proof3}
	lead us to the intended contradiction.
\end{proof}

As we have done for the case of Neumann conditions,
now that we have the ``concentration'' result ensured by 
Lemma~\ref{lem:in-a-small-sphere-Dirichlet},
it is possible to prove that if $m$ and $\varepsilon$ are sufficiently small
for any $u \in \overline{\mathcal{E}}_{\varepsilon,m}^{\sigma \overline{I}_{M}(m) + \tau(m)}$
we have that $\beta^*(u)$ is in a neighbourhood of $M\subset \mathbb{R}^{\tilde{n}}$.
More formally,
setting 
\[
	M_r = \left\{x \in \mathbb{R}^{\tilde{n}}: \dist_{\mathbb{R}^{\tilde{n}}}(x,M) < r\right\},
	\qquad \forall r > 0,
\]
we have the following result.
\begin{lemma}
	\label{lem:beta*-near-M-Dirichlet}
	Assume that \ref{asu_nondegenerate} and \ref{asu_coercive} hold. For any $r > 0$,
	there exists 
	$m_2 = m_2(M,g,r,\mathrm{diam}_{\mathbb{R}^{\tilde{n}}}(M)) > 0$
	such that for every $m \in (0,m_2)$
	there exists another positive constant
	$\varepsilon_2 = \varepsilon_2(M,g,r,m)$
	such that for every $\varepsilon \in (0,\varepsilon_2)$
	and any $u \in \overline{\mathcal{E}}_{\varepsilon,m}^{\sigma \overline{I}_M(m) + \tau(m)}$
	we have $\beta^*(u)\in M_r$.
\end{lemma}
\begin{proof}
	Cf. Lemma~\ref{lem:beta*-near-partialM}.
\end{proof}
\begin{remark}
	\label{rem:def-r1-Dirichlet}
	By the previous lemma and the compactness of $M$,
	there exists $r_1 > 0$
	such that for every $r \in ]0,r_1[$,
	setting $m_2 = m_2(M,g,r,\mathrm{diam}_{\mathbb{R}^{\tilde{n}}}(M))$
	and $\varepsilon_2 = \varepsilon_2(M,g,r,m)$ as in the lemma,
	the map 
	$\mathcal{B}= \pi_M\circ\beta^* \colon
	\overline{\mathcal{E}}_{\varepsilon,m}^{c}\to M$
	is well defined and continuous,
	with $c = \sigma \overline{I}_M(m) + \tau(m)$.
\end{remark}

\subsection{Conclusion of the proof}
In this section we show that if $m$ and $\varepsilon$
are sufficiently small then
$\mathcal{B} \circ \mathcal{P}_{\varepsilon,m} \colon M \to M$
is homotopic to the identity map.
Then, Theorem~\ref{theorem:ACH-Dirichet}
follows as an application of Theorem~\ref{theorem:photography}.

Recalling the definition of $\delta_M > 0$
given in Section~\ref{sec:photo-Dirichlet}
for the purpose of defining the photography map
$\mathcal{P}_{\varepsilon, m}\colon M \to \mathcal{H}_{m,0}$,
let us also define the map
$\widetilde{\mathcal{P}}_{\varepsilon,m}\colon M^{\delta_M/2} \to \mathcal{H}_{m,0}$
as follows:
\begin{equation}
	\label{eq:def-tilde-photo-Dirichlet}
	\widetilde{\mathcal{P}}_{\varepsilon,m}(p)
 = u_{\varepsilon,p}^m,
\end{equation}
where 
$M^{\delta_M/2} = \{p \in M: \dist_g(p,\partial M)\ge \delta_M/2\}$.
With this notation, $\mathcal{P}_{\varepsilon,m}$
is given by $\widetilde{\mathcal{P}}_{\varepsilon,m}\circ\mathcal{L}$.

\begin{lemma}
	\label{lem:barycenter+photography-Dirichlet}
	Assume that \ref{asu_nondegenerate} and \ref{asu_coercive} hold. For every $r \in ]0,\delta_M/2[$
	there exists $m_3 = m_3(M,g,r) > 0$
	such that for every $m \in (0,m_3)$
	there exists $\varepsilon_3 = \varepsilon_3(M,g,r,m) > 0$
	such that
	\begin{equation}
		\label{eq:bound-on-dist-beta-photo-Dirichlet}
		\norm{\beta^*(\widetilde{\mathcal{P}}_{\varepsilon,m}(p)) - p}_{\mathbb{R}^{\tilde{n}}}
		\le r,
		\qquad \forall p \in M.
	\end{equation}
\end{lemma}
\begin{proof}
	Cf. Lemma~\ref{lem:barycenter+photography}.
\end{proof}

\begin{lemma}
	\label{lem:homotopy-barycenter+photography-Dirichlet}
	Assume that \ref{asu_nondegenerate} and \ref{asu_coercive} hold,
    let $\mathcal{P}_{\varepsilon,m}\colon M \to \mathcal{H}_{m,0}$
	be given by Definition~\ref{def:photography-Dirichlet}
	   and let $\mathcal{B}\colon H^1(M,\mathbb{R}) \to M$
	be defined as $\pi_M \circ \beta^*$.
	There exists $m^* = m^*(M,g,W) > 0$ such that 
	for any $m \in (0, m^*)$ there exists
	$\varepsilon^* = \varepsilon^*(M,g,W,m) > 0$ such that
	for any $\varepsilon \in (0, \varepsilon^*)$
	the composition map 
	\[
		\mathcal{B} \circ \mathcal{P}_{\varepsilon,m}
		\colon M \to M
	\]
	is well defined and homotopic to the identity map.
\end{lemma}
\begin{proof}
	The proof is very similar to the one of Lemma~\ref{lem:homotopy-barycenter+photography},
	but it has to take into account the action of the boundary layer map $\mathcal{L}$.
	However, 
	since $\mathcal{L}$ is homotopic to the identity map, 
	it suffices to show that 
	$\mathcal{B}\circ\widetilde{\mathcal{P}}_{\varepsilon,m}\colon M^{\delta_M/2} \to M$
	is homotopic to the identity map as well,
	provided that $m$ and $\varepsilon$ are sufficiently small.

	By the compactness of $M$, there exists a positive constant
	$C_M$ such that 
	$\dist_g(p,q)\le C_M\norm{p-q}_{\mathbb{R}^{\tilde{n}}}$,
	for every $p,q \in M$.
	We denote by
	$\mathrm{inj}( M_{\delta_M/2})$
	the \emph{injectivity radius} of $M_{\delta_M/2}$ in $M$, that is
	\begin{equation*}
		\mathrm{inj}(M^{\delta_M/2})\coloneqq
		\inf_{p \in M^{\delta_M/2}}\sup\Big\{R>0
			\text{ s.t. }
			\mathrm{exp}^{M}_{p}\colon B^{M}_{g}(p,{R})
		\to M \mbox{ is a diffeomorphism}\Big\}.
	\end{equation*}
	Note that by this construction we have
	$\mathrm{inj}( M^{\delta_M/2}) \le \delta_M/2$.

	Recalling the definition of $r_1$ given in 
	Remark~\ref{rem:def-r1-Dirichlet}, we set
	\[
		r^* = \frac{1}{2}\min\left\{\frac{\delta_M}{C_M},\mathrm{inj}(M_{\delta_M/2}),r_1 \right\}
		\le \frac{\delta_M}{4}.
	\]
	Let 
        $m_1 = m_1(M,g,W,\tau)$,
        $m_2 = m_2(M,g,r^*,\mathrm{diam}_{\mathbb{R}^{\tilde{n}}}(M))$ and
	$m_3 = m_3(M,g,r^*)$
	be the positive constants given by 
 Proposition~\ref{prop:photo-sublevel-Dirichlet}, Lemma~\ref{lem:beta*-near-M-Dirichlet}
	and Lemma~\ref{lem:barycenter+photography-Dirichlet}, respectively.
	Let $m^* = \min\{m_1,m_2,m_3\} > 0$
	and for every $m \in (0,m^*)$ 
	let $\varepsilon^* = \min\{\varepsilon_1,\varepsilon_2,\varepsilon_3\}$,
	where $\varepsilon_i$ is defined by using the          same results
	of $m_i$, for $i = 1,2,3$.
	Since $m < m_1$,
	for every $p \in M_{\delta_M/2}$ we have
	that $\widetilde{\mathcal{P}}_{\varepsilon,m}(p) \in \overline{\mathcal{E}}_{\varepsilon,m}^c$,
	with $c = \sigma \overline{I}_M(m) + \tau(m)$.
	Since  $m < m_3$,
	this implies that
	$\norm{\beta^*(\widetilde{\mathcal{P}}_{\varepsilon,m}(p)) - p}_{\mathbb{R}^{\tilde{n}}}\le r^*$
	for every $p \in M_{\delta_M/2}$.
	Therefore, we have
	\begin{multline}
		\label{eq:homotopy-barycenter+photography-Dirichlet-proof1}
		\mathrm{dist}_g\big(\mathcal{B}(\widetilde{\mathcal{P}}_{\varepsilon,m}(p)),p\big)
		\le C_{ M}\norm{\pi_{ M}
			(\beta^*
		(\widetilde{\mathcal{P}}_{\varepsilon,m}(p)))-p}_{\mathbb{R}^{\tilde{n}}}\\
		\le C_{M}\big(
			\norm{\pi_{M}(\beta^*(\widetilde{\mathcal{P}}_{\varepsilon,m}(p))) -
			\beta^*(\widetilde{\mathcal{P}}_{\varepsilon,m}(p))}_{\mathbb{R}^{\tilde{n}}}
			+ 
			\norm{\beta^*(\widetilde{\mathcal{P}}_{\varepsilon,m}(p))
			- p}_{\mathbb{R}^{\tilde{n}}}
		\big)\\
		\le 2 C_M r^* \le \frac{\delta_M}{2},
	\end{multline}
	so the map
	$F\colon [0,1] \times M_{\delta_M/2} \to M$,
	given by
	\begin{equation*}
		F(t,p)\coloneqq
		\mathrm{exp}^{M}_p\Big(t
			\big(\mathrm{exp}^{M}_p\big)^{-1}
		\big(\mathcal{B}(\widetilde{\mathcal{P}}_{\varepsilon,m}(p))\big)\Big),
	\end{equation*}
	is well-defined,
	continuous and gives a homotopy equivalence between
	$\mathcal{B}\circ \widetilde{\mathcal{P}}_{\varepsilon,m}$ and
	the identity map in $M^{\delta_M/2}$.
\end{proof}

\begin{proof}[Proof of Theorem~\ref{theorem:ACH-Dirichet}]
	As previously stated, 
	this proof is based on Theorem~\ref{theorem:photography}.
	Indeed, for every $m,\varepsilon > 0$ the functional 
	$\overline{\mathcal{E}}_{\varepsilon,m}$ is bounded below and
	it satisfies the Palais-Smale condition
	(see Lemma~\ref{lem:E_eps_C2-PS}).
	Moreover, let $m^*> 0$ be given by 
	Lemma~\ref{lem:homotopy-barycenter+photography-Dirichlet}
	and for every $m \in (0,m^*)$
	let $\varepsilon^* > 0$ be given by the same lemma.
	Setting $c =\sigma \overline{I}_M(m) + \tau(m)$,
	where $\tau\colon \mathbb{R}^+ \to \mathbb{R}$ is defined
	in Proposition~\ref{prop:photo-sublevel-Dirichlet},
	we establish that the photography map
	$\mathcal{P}_{\varepsilon,m}\colon M \to \overline{\mathcal{E}}_{\varepsilon,m}^c$
	(as defined in Definition~\ref{def:photography-Dirichlet})
	and the barycenter map $\mathcal{B}\colon \overline{\mathcal{E}}_{\varepsilon,m}^c \to M$
	(defined as $\pi_M \circ \beta^*$) are well-defined,
 for every $\varepsilon \in (0,\varepsilon^*)$.
	Furthermore, according to Lemma~\ref{lem:homotopy-barycenter+photography-Dirichlet},
	their composition $\mathcal{B}\circ \mathcal{P}_{\varepsilon,m}$
	is homotopic to the identity map.
	Then, all the conclusions of Theorem~\ref{theorem:ACH-Dirichet}
	can be derived by applying Theorem~\ref{theorem:photography}.
\end{proof}

\section{Generic non-degeneracy}
\label{sect_generic}
The conclusions of Theorems \ref{theorem:ACH-Neumann} and \ref{theorem:ACH-Dirichet} are stronger if one can ensure that for given $m$ and $\varepsilon$ all the solutions of \eqref{eq:PDE_Neumann} and \eqref{eq:PDE_Dirichlet} are nondegenerate. In many situations, nondegeneracy of solutions is a generic property, in the sense that it is obtained after some arbitrary small perturbation of a key parameter of the problem. This is the case for semilinear elliptic equations as the ones considered in this paper, meaning that the stronger existence statements of Theorem \ref{theorem:ACH-Neumann} and \ref{theorem:ACH-Dirichet}, based on Morse inequalities, hold in generic situations. The purpose of this section is to give some insight on how these genericity properties are obtained from the available literature, which contains analogous results in similar settings. 

\subsection{Abstract transversality result}
Generic nondegeneracy is usually obtained by applying Sard-Smale Theorem \cite{smale65} and abstract transversality results on infinite-dimensional Banach manfiolds, see Quinn \cite{quinn70}, Saut and Temam \cite{saut-temam79} and Uhlenbeck \cite{uhlenbeck76}. Following \cite{saut-temam79}, we now introduce some objects that will be fixed for the rest of this section and recall some standard definitions. Let $X$, $Y$ and $Z$ be Banach spaces and $U \subset X$, $V \subset Y$ open subsets. Let $F:U \times V \to Z$ be of class $C^k$ for $k \geq 1$. We will denote by $DF(x_0,y_0)$ the (total) differential of $F$ at the point $(x_0,y_0) \in U \times V$ and by $D_XF(x_0,y_0)$, $D_YF(x_0,y_0)$ the differentials with respect to the $x$ and the $y$ components, respectively. Given $Z'$ a Banach space, $U' \subset Z'$ an open subset and $g: Z' \to Z$ of class $C^1$, we say that $z_0 \in Z$ is a regular value of $g$ if $Dg(z_0')$ (the differential of $g$ at $z_0'$)  is onto for any $z_0' \in U'$ such that $g(z_0')=z_0$. The abstract result we use is the following:
\begin{theorem}[Theorem 1.1 in \cite{saut-temam79}]\label{theorem_transversality}
Assume that $z_0 \in Z$ is a regular value of $F$ and that, moreover:
\begin{enumerate}
\item For any $x \in U$ and $y \in V$, $D_XF(x,y): X \to Z$ is a Fredholm map of index $l<k$.
\item The set of $x \in U$ such that $F(x,y)=z_0$ for all $y$ in a compact subset of $V$ is a relatively compact set of $U$.
\end{enumerate}
Then, the set
\begin{equation*}
\mathcal{O}\coloneqq \{ y \in V: z_0 \mbox{ is a regular value of } F(\cdot,y),
\end{equation*}
is a dense open subset of $V$.
\end{theorem}
\subsection{Genericity with respect to the Riemannian metric}
Theorem \ref{theorem_transversality} was used in \cite{andrade-conrado-nardulli-piccione-resende} (cf. also \cite{micheletti-pistoia09}) in order to prove generic non-degeneracy for solutions of \eqref{eq:Cahn-Hilliard-PDE} (i.e., critical points of the $C^2$ functionals $\CE_{\eps,m}$ and $\overline{\CE}_{\eps,m}$) with respect to the Riemannian metric. The analogous result holds in our setting: 
\begin{theorem}\label{theorem_genericity_metric}
Let $\mathrm{Met}^\infty(M)$ be the space of $C^\infty$ Riemannian metrics on $M$ and $m \in \R$. For $\hat{g} \in \mathrm{Met}^\infty(M)$ and $\eps \in (0,+\infty)$, we denote by $\CE_{\eps,m,\hat{g}}$ and $\overline{\CE}_{\eps,m,\hat{g}}$ the functionals $\CE_{\eps,m}$ and $\overline{\CE}_{\eps,m}$ with respect to the metric $\hat{g}$. Then, for any $g_0 \in \mathrm{Met}^\infty(M)$ the sets
\begin{align*}
\left\{
\begin{aligned}
(\eps,\hat{g}) \in (0,+\infty) & \times \mathrm{Met}^\infty(M): 
\mbox{ any critical point }\\
& \mbox{ of } \CE_{\eps,m,\hat{g}} \mbox{ in } H^1_{g_0}(M,\R) \mbox{ is non-degenerate}
\end{aligned} 
\right\}
\end{align*}
and
\begin{align*}
\left\{
\begin{aligned}
(\eps,\hat{g}) \in (0,+\infty) & \times \mathrm{Met}^\infty(M): 
\mbox{ any critical point }\\
& \mbox{ of } \overline{\CE_{\eps,m,\hat{g}}} \mbox{ in } H^1_{g_0}(M,\R) \mbox{ is non-degenerate}
\end{aligned} 
\right\}
\end{align*}
are open dense subsets of $(0,+\infty) \times \mathrm{Met}^\infty(M)$.
\end{theorem}
The proof of Theorem \ref{theorem_genericity_metric} is obtained as an application of the abstract Theorem \ref{theorem_transversality}.
The interested reader can find more details in \cite{andrade-conrado-nardulli-piccione-resende,micheletti-pistoia09}. 

\subsection{Genericity with respect to the domain}
Notice that Theorem \ref{theorem_genericity_metric} has a drawback: In some applications, one might not want to modify the metric of the problem. For instance, if $M$ is an Euclidean domain $\Omega \subset \R^n$ then it would not make much sense to replace the Euclidean metric on $\Omega$ by another one.
It is therefore reasonable to be intereseted in generic nondegeneracy with respect to a different key parameter. It might be tempting to consider a genericity result based on modifying the potential instead of the Riemannian metric.
However, Brunovsk\'{y} and Pol\'{a}\v{c}ik~\cite{brunovsky-polacik} produced a counterexample for semilinear elliptic equations on the Euclidean ball with homogeneous Dirichlet boundary conditions. Our setting is slightly different than theirs (in particular, more restrictive) meaning that one could still hope to obtain genericity results with respect to the potential. However, the investigation of this question goes beyond the scope of this paper. Instead, we will present a genericity result based on perturbation of the domains, obtained by adapting a result due to Saut and Temam \cite{saut-temam79}, see also the book by Henry \cite{henry05}. This can be somehow understood as a Euclidean counterpart of Theorem \ref{theorem_genericity_metric}.

Following \cite{henry05}, let us give the precise statement of the genericity result. Fix $\Omega \subset \R^n$ a bounded and smooth domain. For $l \in \{1,\ldots,+\infty\}$, let $\mathrm{Diff}^l(\Omega)$ be the set of maps in $\CC^l(\Omega,\R^n)$ such that $h\colon\Omega \to h(\Omega)$ is a $C^l$-diffeomorphism. In particular, $\mathrm{Id} \in \mathrm{Diff}^l(\Omega)$. For $l<+\infty$, let us endow $\CC^l(\Omega,\R^n)$ with the usual sup norm, so it is a Banach space. This induces the usual Whitney topology on $\CC^\infty(\Omega,\R^n)$. Then, with these choices of topologies, $\mathrm{Diff}^l(\Omega)$ is open in $\CC^l(\Omega,\R^n)$. For $h \in \mathrm{Diff}^l(\Omega)$ and $\eps>0$, let $\CE_{\eps,h}$ denote the energy functional with parameter $\eps$ in the domain $h(\Omega)$. For $m \in \R$, define $\CE_{\eps,m,h}$ and $\overline{\CE}_{\eps,m,h}$ as the restriction of $\CE_{\eps,h}$ to $\CH_m$ and $\CH_{m,0}$ respectively. The main result reads as follows:
\begin{theorem}\label{THEOREM_genericity_domain}
The sets
\begin{equation*}
\big\{ (\eps,h) \in (0,+\infty) \times \mathrm{Diff}^{\infty}(\Omega): \mbox{any critical point of } \CE_{\eps,m,h} \mbox{ is nondegenerate}\big\}
\end{equation*}
and
\begin{equation*}
\big\{ (\eps,h) \in (0,+\infty) \times \mathrm{Diff}^{\infty}(\Omega): \mbox{any critical point of } \overline{\CE}_{\eps,m,h} \mbox{ is nondegenerate}\big\}
\end{equation*}
are dense in $(0,+\infty) \times \mathrm{Diff}^\infty(\Omega)$.
\end{theorem}
The proof of Theorem \ref{THEOREM_genericity_domain} is based on the transversality Theorem \ref{theorem_transversality}, and follows by combining the ideas in \cite{andrade-conrado-nardulli-piccione-resende} with those in \cite{saut-temam79,henry05}. This result is applied to $\mathrm{Diff}^l(\Omega,\R^n)$, which is an open subset of the Banach space $\CC^l(\Omega,\R^n)$. As a consequence, one obtains that the sets
\begin{equation*}
\{ (\eps,h) \in (0,+\infty) \times \mathrm{Diff}^{l}(\Omega): \mbox{Any critical point of } \CE_{\eps,m,h} \mbox{ is nondegenerate}\}
\end{equation*}
and
\begin{equation*}
\{ (\eps,h) \in (0,+\infty) \times \mathrm{Diff}^{l}(\Omega): \mbox{Any critical point of } \overline{\CE}_{\eps,m,h} \mbox{ is nondegenerate}\}
\end{equation*}
are open and dense in $(0,+\infty) \times \mathrm{Diff}^\infty(\Omega)$. Theorem \ref{THEOREM_genericity_domain} then follows by Baire category Theorem, by taking the intersections $\cap_{l \in \N^*}$ of the sets above. We skip the details.

\section{A partial result without the subcritical growth assumption}\label{sect_subcritical}
The sole purpose of the subcritical growth assumption \ref{asu_subcritical} is to ensure that the variational problems under consideration are compact (more precisely, that the functionals $\CE_{\eps,m}$ and $\overline{\CE}_{\eps,m}$ satisfy the Palais-Smale condition). However, as it was already observed in \cite{benci-nardulli-osorio-piccione}, it is possible to drop \ref{asu_subcritical} and obtain weaker versions of Theorems \ref{theorem:ACH-Neumann} and \ref{theorem:ACH-Dirichet}. More precisely, one has:
\begin{theorem}\label{theorem:ACH-Neumann-weak}
Assume that \ref{asu_nondegenerate} and \ref{asu_coercive} hold. Then, there exists $m^*>0$ such that for all $m \in (0,m^*)$ there exist $\eps_{m},c_{m}>0$ such that for any $\eps \in (0,\eps_m)$ the Neumann problem \eqref{eq:PDE_Neumann} has at least $\mathrm{cat}(\partial M)$ solutions $u_{\eps,m}$ with $\CE_{\eps}(u_{\eps,m}) \leq c_{m}$. Moreover, if $\eps$ and $m$ as above are such that all critical points of $\CE_{\eps,m}$ are non-degenerate, then \eqref{eq:PDE_Neumann} has at least $\mathscr{P}_1(\partial M)$ solutions $u_{\eps,m}$ with $\CE_{\eps}(u_{\eps,m}) \leq c_{m}$.
\end{theorem}
\begin{theorem}\label{theorem:ACH-Dirichlet-weak}
Assume that \ref{asu_nondegenerate} and \ref{asu_coercive} hold. Then, there exists $m^*>0$ such that for all $m \in (0,m^*)$ there exist $\eps_{m},c_{m}>0$ such that for any $\eps \in (0,\eps_m)$ the Dirichlet problem \eqref{eq:PDE_Dirichlet} has at least $\mathrm{cat}(M)$ solutions $u_{\eps,m}$ with $\CE_{\eps}(u_{\eps,m}) \leq c_{m}$. Moreover, if $\eps$ and $m$ as above are such that all critical points of $\overline{\CE}_{\eps,m}$ are non-degenerate, then \eqref{eq:PDE_Dirichlet} has at least $\mathscr{P}_1(M)$ solutions $u_{\eps,m}$ with $\CE_{\eps}(u_{\eps,m}) \leq c_{m}$.
\end{theorem}
The proofs of Theorems \ref{theorem:ACH-Neumann-weak} and \ref{theorem:ACH-Dirichlet-weak} work essentially like that of \cite[Theorem 5.9]{benci-nardulli-osorio-piccione}.  However, the assumptions on the potential that we take in this paper are slightly different than those in \cite{benci-nardulli-osorio-piccione}. Therefore, we include the proofs of Theorems \ref{theorem:ACH-Neumann-weak} and \ref{theorem:ACH-Dirichlet-weak} for completeness. The crucial ingredient is that for small $\eps$ one can obtain a priori bounds on the $L^\infty$ norm of any solution $(u_\eps,\lambda_\eps)$ of the equation
\begin{equation}
\label{equation_ub}
-\eps \Delta u_{\eps}+\frac{1}{\eps}\hat{W}'(u_\eps)=\lambda_\eps \mbox{ on } M.
\end{equation}
according to its energy
(see Proposition \ref{PROPOSITION_ub} below).
Such ingredient allows to modify potentials close to infinity to turn them into potentials with subcritical growth at infinity to which one may apply Theorems \ref{theorem:ACH-Neumann} and \ref{theorem:ACH-Dirichet}. The a priori bounds imply that the low energy solutions of the modified potential are also solutions of the original potential. However, no such thing can be said at this point regarding high energy solutions, hence the weaker statements of Theorems \ref{theorem:ACH-Neumann-weak} and \ref{theorem:ACH-Dirichlet-weak}. 
\begin{proposition}\label{PROPOSITION_ub}
Let $\hat{W} \in \CC^3_{\mathrm{loc}}(\R,[0,+\infty])$ be a double-well potential vanishing exactly on $\{0,1\}$ and satisfying \ref{asu_nondegenerate} and \ref{asu_coercive}.
There exists $\eps_{ub}>0$ such that for any $\eps \in (0,\eps_{\mathrm{ub}}]$ and 
any $(u_{\eps},\lambda_{\eps})$ solution of~\eqref{equation_ub}
such that $\CE_{\eps,\hat{W}}(u_{\eps})<+\infty$
we have
\begin{equation*}
\lVert u_\eps \rVert_{L^\infty(M,\R)}\leq C_{\mathrm{ub}}(\eps_{ub},\CE_{\eps,\hat{W}}(u_{\eps})),
\end{equation*}
with $C_{\mathrm{ub}}(\eps_{ub},\CE_{\eps,\hat{W}}(u_{\eps}))>0$ a constant depending on the quantities $\eps_{ub}$ and $\CE_{\eps,\hat{W}}(u_{\eps})$ but independent on $\eps$.
\end{proposition}
The proof of Proposition \ref{PROPOSITION_ub} relies on the following result.
\begin{lemma}\label{LEMMA_chen_manifold}
Let $\hat{W} \in \CC^3_{\mathrm{loc}}(\R,[0,+\infty])$ be a double-well potential vanishing exactly on $\{0,1\}$ and satisfying \ref{asu_nondegenerate} and \ref{asu_coercive}. There exist $\eps_{\mathrm{est}}>0$ and $C_{\mathrm{est}}>0$ such that for any $\eps \in (0,\eps_{\mathrm{est}})$, $m \in [0,1]$ and $(u_{\eps},\lambda_{\eps})$ a solution of \eqref{equation_ub} we have
\begin{equation*}
\lvert \lambda_{\eps} \rvert \leq C_{\mathrm{est}}\CE_{\eps,\hat{W}}(u_{\eps}).
\end{equation*}
\end{lemma}
Lemma \ref{LEMMA_chen_manifold} was proven by Chen in \cite[Lemma 3.4]{chen96} in the case of an Euclidean domain of $\R^n$ and then extended to the setting of closed manifolds in \cite[Proposition 5.3]{benci-nardulli-osorio-piccione}.
In both cases, the proof is obtained by combining standard elliptic estimates with mollification arguments.
As the arguments carry on directly to our setting, we skip the proof.
We are now ready to prove Proposition \ref{PROPOSITION_ub}, which relies on standard arguments (see, for instance, \cite[Theorem 5.9]{benci-nardulli-osorio-piccione}).
\begin{proof}[Proof of Proposition \ref{PROPOSITION_ub}]
Assume that $\eps \in (0,\eps_{\mathrm{est}})$, with $\eps_{\mathrm{est}}$ as in Lemma \ref{LEMMA_chen_manifold}.
By \ref{asu_coercive}, we find $u^*(\eps_{\mathrm{est}},\CE_{\eps,\hat{W}}(u_\eps))$ depending only on $\eps_{\mathrm{est}}>0$ and $\CE_{\eps,\hat{W}}(u_\eps)$ such that
\begin{equation*}
\frac{1}{\eps}\hat{W}'(u)>C_{\mathrm{est}}\CE_{\eps,\hat{W}}(u_{\eps}),
\qquad\text{ for all } u \geq u^*(\eps_{\mathrm{est}},\CE_{\eps,\hat{W}}(u_\eps))
\end{equation*}
and 
\begin{equation*}
\frac{1}{\eps}\hat{W}'(u)<-C_{\mathrm{est}}\CE_{\eps,\hat{W}}(u_{\eps}),
\qquad \text{ for all } u \leq u^*(\eps_{\mathrm{est}},\_{\eps}(u_\eps)),
\end{equation*}
with $C_{\mathrm{est}}>0$ as in Lemma \ref{LEMMA_chen_manifold}. 
Let $x_{\mathrm{max}} \in M$ be a maximum point for $u_{\eps}$ and assume by contradiction that $\lVert u_\eps \rVert_{L^\infty(M,\R)} >u^*(\eps_{\mathrm{est}},\CE_{\eps,\hat{W}}(u_\eps))$. Then, using \eqref{equation_ub} and Lemma \ref{LEMMA_chen_manifold} it follows
\begin{equation*}
-\eps \Delta u_{\eps}(x_{\mathrm{max}})=\lambda_\eps - \frac{1}{\eps}\hat{W}'(u_\eps(x_{\mathrm{max}})) < 0,
\end{equation*}
which gives the contradiction. Hence, one has that $\max_M u \leq u^*(\eps_{\mathrm{est}},\CE_{\eps,\hat{W}}(u_\eps))$ and, in an analogous fashion, one finds that $\min_M u \geq u^*(\eps_{\mathrm{est}},\CE_{\eps,\hat{W}}(u_\eps))$. This establishes the result by choosing $\eps_{\mathrm{ub}}=\eps_{\mathrm{est}}$ and $C_{\mathrm{ub}}(\eps_{ub},\CE_{\eps,\hat{W}}(u_{\eps}))=u^*(\eps_{\mathrm{est}},\CE_{\eps,\hat{W}}(u_\eps))$.
\end{proof}

At this point, the proof of Theorem \ref{theorem:ACH-Neumann-weak} can be completed as follows.
\begin{proof}[Proof of Theorem \ref{theorem:ACH-Neumann-weak}]
Let $\eps \in (0,\eps_{\mathrm{ub}})$, with $\eps_{\mathrm{ub}}$ as in Lemma \ref{LEMMA_chen_manifold} and $C_{\mathrm{ub}}(\eps_{\mathrm{ub}},1)>0$ the constant given by Proposition \ref{PROPOSITION_ub}.
We can find $u^* \geq C_{\mathrm{ub}}(\eps_{\mathrm{ub}},1)$ and $\hat{W} \in \CC_{\mathrm{loc}}^3(\R,[0,+\infty))$ such that $\hat{W}^{-1}(0)=\{0,1\}$ which satisfies \ref{asu_nondegenerate}, \ref{asu_coercive} and \ref{asu_subcritical} with $\hat{W}(u)=W(u)$ for all $u \in [-u^*,u^*]$.
We now apply Theorem \ref{theorem:ACH-Neumann} to $\hat{W}$.
In particular, there exists $\hat{m}^*>0$ such that for all $m \in (0,\hat{m}^*)$ there exist $\hat{\eps}_m,\hat{c}_m >0$ such that for all $\eps \in (0,\hat{\eps}_m)$ we have that there exist $\cat(\partial M)$ solutions of \eqref{eq:PDE_Neumann} (for $\hat{W}$) $(\hat{u}_{\eps,m},\hat{\lambda}_{\eps,m})$ with $\CE_{\eps,\hat{W}}(\hat{u}_{\eps,m}) \leq \hat{c}_m$.
Moreover $\hat{c}_m$ can be chosen such that $\hat{c}_m \to 0$ as $m \to 0$, see the end of Section \ref{sect_Neumann_completed}.
Therefore, choose $m^* \leq \hat{m}^*$ such that $\hat{c}_m \leq 1$ for all $m \in (0,m^*)$.
For all such $m$, let $\eps_m\coloneqq\frac{1}{2}\min\{\hat{\eps}_m,\eps_{\mathrm{est}}\}>0$, $\eps \in (0,\eps_m)$ and $\hat{u}_{\eps,m}$.
Since $\CE_{\eps}(\hat{u}_{\eps,m}) \leq 1$ whenever $\hat{u}_{\eps,m}$ is a solution as above, we can apply Proposition \ref{PROPOSITION_ub} and find that $\lVert \hat{u}_{\eps,m} \rVert_{L^\infty(M,\R)} \leq C_{\mathrm{ub}}(\eps_{\mathrm{ub}},1) \leq u^*$ which implies that $\hat{u}_{\eps,m}$ is a solution of \eqref{eq:PDE_Neumann} for $W$ and hence the result. 
\end{proof}
The proof of Theorem~\ref{theorem:ACH-Dirichlet-weak} works in the same way, so we skip it.
\printbibliography

\bigskip
\noindent
(D. Corona)
\textsc{
School of Science and Technology,
Universit\`a degli Studi di Camerino}\\
Via Madonna delle Carceri 9,
62032, Camerino (MC), Italy\\
\emph{E-mail}: {\tt dario.corona@unicam.it}

\bigskip
\noindent
(S. Nardulli)
\textsc{
Centro de Matem\'atica Cogni\c{c}\~ao Computa\c{c}\~ao,
Universidade Federal do ABC}\\
Avenida dos Estados, 5001,
Santo Andr\'e, SP, CEP 09210-580, Brazil\\
\emph{E-mail}: {\tt stefano.nardulli@ufabc.edu.br}

\bigskip
\noindent
(R. Oliver-Bonafoux)
\textsc{
Dipartimento di Informatica,
Universit\`a di Verona}\\
Strada Le Grazie 15,
37134,Verona, Italy\\
\emph{E-mail}: {\tt ramon.oliverbonafoux@univr.it}

\bigskip
\noindent
(G. Orlandi)
\textsc{
Dipartimento di Informatica,
Universit\`a di Verona}\\
Strada Le Grazie 15,
37134,Verona, Italy\\
\emph{E-mail}: {\tt giandomenico.orlandi@univr.it}

\bigskip
\noindent
(P. Piccione)
\textsc{
Departamento de Matem\'atica,
Universidade de S\~ao Paulo}\\
Rua do Mat\~ao 1010,
S\~ao Paulo, SP 05508--090, Brazil \\
\emph{E-mail}: {\tt paolo.piccione@usp.br}


\end{document}